\documentclass[11pt, twoside]{article}
\usepackage{amscd,amssymb}
\usepackage{amsthm,amsmath,amssymb, color}
\usepackage[matrix,arrow]{xy}

\newtheorem{theorem}[equation]{Theorem}

\newtheorem{proposition}[equation]{Proposition}
\newtheorem{lemma}[equation]{Lemma}
\newtheorem{corollary}[equation]{Corollary}
\newtheorem{conjecture}[equation]{Conjecture}

\theoremstyle{definition}
\newtheorem{example}[equation]{Example}
\newtheorem{definition}[equation]{Definition}

\theoremstyle{remark}
\newtheorem{remark}[equation]{Remark}

\newcommand{\mult}{\operatorname{mult}}

\newcommand{\lct}{\operatorname{lct}}

\makeatletter\@addtoreset{equation}{subsection} \makeatother
\renewcommand{\theequation}{\thesubsection.\arabic{equation}}

\begin{document}

\thispagestyle{empty}
 \Large
\begin{center}
\textbf{Log canonical thresholds of certain Fano hypersurfaces}
\end{center}
\vspace{7mm}

\normalsize
\begin{center}
\textbf{Ivan Cheltsov\footnote{School of Mathematics, The
University of Edinburgh,  Edinburgh, EH9 3JZ, UK;
 \texttt{I.Cheltsov@ed.ac.uk}} and Jihun Park\footnote{IBS Center for Geometry and Physics;
 Department of
Mathematics, POSTECH, Pohang, Kyungbuk 790-784, Korea;
\texttt{wlog@postech.ac.kr}}
}\\

\end{center}

\vspace{7mm}

 \small {\noindent\textbf{Abstract.} We study  log canonical thresholds on quartic threefolds, quintic
fourfolds, and double spaces. As an important application, we show
that they have  K\"ahler-Einstein metrics if they are
general.\vspace{3mm}

\vspace{3mm}

}

\vspace{1cm}

\section{Introduction.}
\label{section:intro}
All varieties are defined over $\mathbb{C}$.

\subsection{Introduction.}

The multiplicity of a nonzero polynomial $f\in\mathbb{C}[z_1,\cdots,
z_n]$ at a point $P\in \mathbb{C}^n$  is the nonnegative integer $m$ such
that $f\in\mathfrak{m}_P^m\setminus\mathfrak{m}_P^{m+1}$, where
$\mathfrak{m}_P$ is the maximal ideal of polynomials vanishing at
the point $P$ in $\mathbb{C}[z_1,\cdots, z_n]$. It can be also defined by
derivatives. The multiplicity of $f$ at the point $P$ is the
number
\[\mult_P(f)=\min\left\{m \ \Big|\ \frac{\partial^m f}{\partial^{m_1}
z_1\partial^{m_2} z_2\cdots\partial^{m_n} z_n}(P)\ne0 \right\}.\]

On the other hand, we have a similar invariant that is defined by
integrations. This invariant, which is called the log canonical
threshold of $f$ at the point $P$, is given by

\[c_P(f)=\mathrm{sup}\left\{c\ \Big|\ |f|^{-c}~\text{is locally}~L^2~\text{near the
 point $P\in\mathbb{C}^n$}\right\}.\]
This number appears in many places. For instance, the log
canonical threshold of the polynomial $f$ at the origin  is the
same as the absolute value of the largest root of the
Bernstein-Sato polynomial of $f$.

Even though log canonical threshold was implicitly known and extensively studied under different names by J.~H.~M Steenbrink, A.~Varchenko and so forth,
 it was formally introduced to birational geometry by V. Shokurov in \cite{Sho93} as follows.  Let $X$ be a $\mathbb{Q}$-factorial variety with at worst log canonical singularities,
$Z\subset X$ a closed subvariety, and $D$ an effective
 $\mathbb{Q}$-divisor on $X$. The log canonical threshold
of $D$ along $Z$ is the number
$$c_Z(X,D)=\mathrm{sup}\left\{c\ \Big|\ \mbox{the log pair } (X, cD)
 ~\text{is log canonical  along}~Z\right\}.$$
 For
 the case $Z=X$ we use the notation $c(X,D)$ instead of
 $c_X(X,D)$.
Because log canonicity is a local property, we see that
$$c_Z(X,D)=\inf_{P\in Z} \left\{c_P(X,D)\right\}.$$
If $X=\mathbb{C}^n$ and $D=(f=0)$, then we also
 use the notation $c_0(f)$ for the log canonical threshold of
 $D$ at the origin.

Even though several methods have been invented in order to compute log
canonical thresholds, it is not easy to compute them in general.
However, many problems in birational geometry are related to log canonical
thresholds. The log canonical thresholds play a significant role in the study on birational geometry. They show many interesting
properties~(see \cite{dFEM02}, \cite{dFEM03},   \cite{How01},  \cite{Ko97},  \cite{Ku1}, \cite{Ku2}, \cite{McKePro04},
 \cite{Pa01}, \cite{Pr01d},
\cite{Pr02b},  \cite{Pr03b}).

We occasionally find it useful to consider the smallest value of  log canonical thresholds of effective divisors linearly equivalent to a given  divisor, in particular, an anticanonical divisor (for instance, see \cite{Pa01}).
\begin{definition}
Let $X$ be a $\mathbb{Q}$-factorial Fano variety with at worst log
terminal singularities. For  a natural number $m>0$, we define the
$m$-th global log canonical threshold of $X$ by the number
$$
\mathrm{lct}_{m}\left(X\right)=\mathrm{inf}\ \left\{ c\left(X,\frac{1}{m}H\right)\ \Big|\ H\in \big|-mK_X\big|\right\}.%
$$
Note that the number $\mathrm{lct}_{m}(X)$ is defined to be
$\infty$ if the linear system $|-mK_X|$ is empty.
 Also, we define the global log
canonical threshold of $X$ by the number
$$\mathrm{lct}(X)=\inf_{n\in\mathbb{N}}\ \Big\{ \mathrm{lct}_m(X)\Big\}.$$
\end{definition}
We can immediately see
\[\mathrm{lct}(X)=\mathrm{sup}\left\{\ c \ \left| \
 \aligned
& \mbox{the log pair } (X, cD) \mbox{ is log canonical for }\\
& \mbox{every effective $\mathbb{Q}$-divisor $D$ with } D\equiv -K_X\\
\endaligned\right.\right\}.
 \]

To see the simplest case, let $S$ be a smooth del Pezzo surface.
It follows from \cite[Theorem~1.7]{Ch07b} and
\cite[Section~3]{Pa01} that
\begin{equation}\label{example:del-Pezzo}
\mathrm{lct}\left(S\right)=\mathrm{lct}_{1}(S)=\left\{%
\aligned
&1/3\ \mathrm{when}\ S\cong\mathbb{F}_{1}\ \mathrm{or}\ K_{S}^{2}\in\{7,9\},\\%
&1/2\ \mathrm{when}\ S\cong\mathbb{P}^{1}\times\mathbb{P}^{1}\ \mathrm{or}\ K_{S}^{2}\in\{5,6\},\\%
&2/3\ \mathrm{when}\ K_{S}^{2}=4,\\%
&2/3\ \mathrm{when}\ S\ \mathrm{is\ a\ cubic\ in}\ \mathbb{P}^{3}\ \mathrm{with\ an\ Eckardt\ point},\\%
&3/4\ \mathrm{when}\ S\ \mathrm{is\ a\ cubic\ in}\ \mathbb{P}^{3}\ \mathrm{without\ Eckardt\ points},\\%
&3/4\ \mathrm{when}\ K_{S}^{2}=2\ \mathrm{and}\ |-K_{S}|\ \mathrm{has\ a\ tacnodal\ curve},\\%
&5/6\ \mathrm{when}\ K_{S}^{2}=2\ \mathrm{and}\ |-K_{S}|\ \mathrm{has\ no\ tacnodal\ curves},\\%
&5/6\ \mathrm{when}\ K_{S}^{2}=1\ \mathrm{and}\ |-K_{S}|\ \mathrm{has\ a\ cuspidal\ curve},\\%
&1\ \mathrm{when}\ K_{S}^{2}=1\ \mathrm{and}\ |-K_{S}|\ \mathrm{has\ no\ cuspidal\ curves}.\\%
\endaligned\right.%
\end{equation}

For a quasismooth hypersurface $X$ in
$\mathbb{P}(a_{0},\ldots,a_{4})$ of degree
$\sum_{i=0}^{4}a_{i}-1$, where $a_{0}\leq\cdots\leq a_{4}$, one
can find $\mathrm{lct}(X)>\frac{3}{4}$ for $1936$ values
of~$(a_{0},a_{1},a_{2},a_{3},a_{4})$ (see
\cite[Corollary~3.4]{JoKo01}). Moreover, for  a quasismooth
hypersurface $X$ in $\mathbb{P}(1,a_{1},\ldots,a_{4})$ of
degree~$\sum_{i=1}^{4}a_{i}$ having terminal singularities,  there
are exactly $95$ possible quadruples $(a_{1},a_{2},a_{3},a_{4})$
found~in~\cite{IF00} and \cite{JoKo01}. It follows from
\cite[Theorem~1.3]{Ch07a} that $\mathrm{lct}(X)=1$ if
$$
\left(a_{1},a_{2},a_{3},a_{4}\right)\not\in\Big\{\left(1,1,1,1\right),\left(1,1,1,2\right), \left(1,1,2,2\right), \left(1,1,2,3\right)\Big\}%
$$
and the hypersurface $X$ is sufficiently general.

It is proved that the global log
canonical threshold of a rational homogeneous space of Picard rank
$1$ and Fano index $r$ is $\frac{1}{r}$ (see \cite[Theorem~2]{Hw06b}).

\begin{example}
\label{example:Cheltsov-Park}
Let $X$ be a smooth hypersurface  of degree $n\geq 3$ in
$\mathbb{P}^{n}$.
Then
$$
\mathrm{lct}_{m}\left(X\right)\geq \frac{n-1}{n}%
$$
due to \cite[Theorem~1.3]{Ch01b} and \cite[Theorem~3.3]{ChPa02}.
Furthermore, $\mathrm{lct}_{m}(X)=\frac{n-1}{n}$ if and only if
$X$ contains a cone of dimension $n-2$ (see
\cite[Conjecture~1.5]{Ch01b}, \cite[Corollary~4.10]{ChPa02}, and
\cite[Theorem~0.2]{dFEM03}). The inequality obviously implies that
$$
\mathrm{lct}\left(X\right)\geq \frac{n-1}{n}.%
$$
However, it is shown that
$\mathrm{lct}(X)=1$ if $X$ is general and $n\geq 6$ (see  \cite[Theorem~2]{Pu04d}).
\end{example}

From the results of \cite{Pu04d},  it is natural to expect the
following:
\begin{conjecture}
\label{conjecture:quartic}  The global log canonical thresholds of a general quartic threefold and a general quintic fourfold are $1$.
\end{conjecture}
This conjecture has been proposed for canonical thresholds in
\cite[Conjecture~2]{Pu04d}.

For an evidence of the conjecture, we can consider the first
global log canonical threshold of a general hypersurface. It is
not hard to show that the first global log canonical threshold of
a general hypersurface of degree $n\geq 4$  in
$\mathbb{P}^{n}$ is one (see
Proposition~\ref{proposition:general-1st-lct}). In the case of
smooth quartic threefolds, we can find all the first global log
canonical thresholds (see Proposition~\ref{theorem:auxiliary}).

 For the global log canonical thresholds, we prove the following:

\begin{theorem}
\label{theorem:main} Let $X$ be a general hypersurface of degree $n=4$ or $5$ in $\mathbb{P}^n$.  Then
$$
\mathrm{lct}\left(X\right)\geq\left\{%
\aligned
&\frac{7}{9} \ \  \mbox{ for $n=4$; }\\%
&\\
&\frac{5}{6} \ \  \mbox{ for $n=5$. }\\%
\endaligned\right.%
$$
\end{theorem}

 The global log canonical threshold of a Fano variety is an algebraic counterpart of the
$\alpha$-invariant introduced in \cite{Ti87}. One of the most
interesting applications of the global log canonical thresholds of
Fano varieties is the following result proved in
\cite[p.~549]{DeKo01} (see also
 \cite{Na90} and \cite{Ti87}).

\begin{theorem}
\label{theorem:KE} Let $X$ be an $d$-dimensional Fano variety  with at most quotient
singularities. The variety $X$ has an orbifold K\"ahler--Einstein metric if the inequality
\[
\mathrm{lct}\left(X\right)>\frac{d}{d+1}
\]
holds.
\end{theorem}
The inequality in Example~\ref{example:Cheltsov-Park} is not strong enough to apply Theorem~\ref{theorem:KE} to a smooth hypersurface of degree $n$ in $\mathbb{P}^n$.
 However, we see that  (\ref{example:del-Pezzo}) enables Theorem~\ref{theorem:KE} to  imply
 the existence of a K\"ahler--Einstein metric on a general cubic surface and
  that \cite[Theorem~2]{Pu04d} enables Theorem~\ref{theorem:KE} to  imply the existence of a K\"ahler--Einstein metric on a general hypersurface of degree $n\geq 6$ in $\mathbb{P}^{n}$. Even though Theorem~\ref{theorem:main} is much weaker than Conjecture~\ref{conjecture:quartic}, they are strong enough to imply
the existence of a K\"ahler--Einstein metric. Consequently, we can obtain the following:
\begin{corollary}
\label{corollary:main} A general hypersurface of degree $n\geq 2$ in
$\mathbb{P}^{n}$ has a K\"ahler--Einstein metric.
\end{corollary}
In fact, a smooth conic in $\mathbb{P}^2$ has a K\"ahler-Einstein metric because it is isomorphic to $\mathbb{P}^1$
and the Fubini-Study metric of a projective space is K\"ahler-Einstein.
Furthermore, a smooth cubic surface always admits a K\"ahler--Einstein
metric (see \cite[Section~2]{Ti90}). Meanwhile, it is proved that a K\"ahler--Einstein metric exists
on a smooth hypersurface in $\mathbb{P}^n$  defined by a homogeneous polynomial equation of the form
$z_0^n+z_1^n+f_n(z_2,\cdots, z_n)=0$, where $n\geq 4$ and $f_n$ is a homogeneous polynomial  of degree $n$ in variables $z_2,\cdots, z_n$ (see \cite[Proposition~3.1]{ArGha06}).

Also, in this paper, we will study log canonical thresholds on
double spaces, \emph{i.e.}, double covers of $\mathbb{P}^n$, and
obtain similar results as what we have on Fano  hypersurfaces in
$\mathbb{P}^n$. For instance, we will prove
 that the first global log canonical
threshold of a smooth double space is equal to its global log
canonical threshold (see
Proposition~\ref{proposition:double-spaces-equality}) and that
every smooth double cover of $\mathbb{P}^n$ ramified along a
hypersurface of degree $2n$ admits a K\"ahler-Einstein metric.

Let us close this section by a conjecture inspired by
\cite[Question~1]{Ti90b}.

\begin{conjecture}
\label{conjecture:stabilization} For a smooth Fano variety $X$,
$\mathrm{lct}(X)=\mathrm{lct}_{m}(X)$ for some natural number $m\geq 1$.
\end{conjecture}

\section{Log canonical threshold of a Fano hypersurface}

\subsection{Hypersurface of degree $n$ in $\mathbb{P}^{n}$.}\label{section:general-thresholds}

As we mentioned,  one can consider the first global log canonical
threshold of a general hypersurface of degree $n\geq 4$  in
$\mathbb{P}^{n}$ in behalf of Conjecture~\ref{conjecture:quartic}.
\begin{proposition}\label{proposition:general-1st-lct}
Let $X$ be a general hypersurface of degree $n\geq 4$  in
$\mathbb{P}^{n}$. Then $\mathrm{lct}_{1}(X)=1$.
\end{proposition}
\begin{proof}
Consider the space
$
\mathcal{S}_n=\mathbb{P}^{n}\times \mathbb{P}\left(\mathrm{H}^{0}\left(\mathbb{P}^{n},\mathcal{O}_{\mathbb{P}^{n}}\left(n\right)\right)\right)%
$
with the natural projections
$p\colon\mathcal{S}_n\to\mathbb{P}^{n}$ and $q\colon\mathcal{S}_n\to
\mathbb{P}\left(\mathrm{H}^{0}(\mathbb{P}^n,\mathcal{O}_{\mathbb{P}^{n}}(n))\right)$. Put
$
\mathcal{I}_n=\left\{\left(O, F\right)\in\mathcal{S}_n\ \Big\vert\ F(O)=0\ \mathrm{and}\ F=0\ \mathrm{is\ smooth.}\right\}%
$.

Let $(O, F)$ be a pair in $\mathcal{I}_n$. Suppose that $O=[1:0:\cdots:0]$. Then $F$ can be given by a polynomial of the form
$
z_0^{n-1}z_{n}+z_0^{n-2}q_2(z_1,\cdots,
z_n)+\cdots+z_0q_{n-1}(z_1,\cdots, z_n)+q_n(z_1,\cdots, z_n)\in\mathbb{C}[z_0,\cdots, z_n],
$
where $q_{i}$ is a homogeneous polynomial of degree $i$.

We say that the point $O$ is $\emph{bad}$ on the
hypersurface $F=0$ if  one of the following condition holds:
\begin{itemize}

\item[$(1)$] $q_2(z_1,\cdots, z_{n-1},0)\equiv 0$.

\item[$(2)$] $q_2(z_1,\cdots, z_{n-1},0)=\{l(z_1,\cdots, z_{n-1})\}^2$ for some linear form $l(z_1,\cdots, z_{n-1})$ and if we assume $l(z_1,\cdots, z_{n-1})=z_{n-1}$, either $q_3(z_1,\cdots, z_{n-2},0,0)\equiv 0$ or $q_3(z_1,\cdots, z_{n-2},0,0)=\{m(z_1,\cdots, z_{n-2})\}^3$ for some linear form $m(z_1,\cdots, z_{n-2})$.
\end{itemize}

Then consider a subset of $\mathcal{I}_n$,  $$
\mathcal{Y}_n=\left\{\left(O, F\right)\in\mathcal{I}_n\ \Big\vert\ \mbox{the point } O\ \mathrm{is\ bad\ on\ the\ quartic}\ F=0\right\}.$$

One can see that for a given point $P$ on $\mathbb{P}^n$, the
dimension of $p^{-1}(P)\cap \mathcal{Y}_n$ is strictly smaller than $h^0
(\mathbb{P}^n,\mathcal{O}_{\mathbb{P}^{n}}(n))-(n+1)$, and hence the
dimension of the space $\mathcal{Y}_n$ is smaller than the dimension
of $\mathbb{P}\left(\mathrm{H}^{0}\left(\mathbb{P}^{n},
\mathcal{O}_{\mathbb{P}^{n}}\left(n\right)\right)\right)$. Therefore,
the image of the regular map $
q\vert_{\mathcal{Y}_n}\colon\mathcal{Y}_n\to
\mathbb{P}\left(\mathrm{H}^{0}\left(\mathbb{P}^{n},
\mathcal{O}_{\mathbb{P}^{n}}\left(n\right)\right)\right) $ is a
proper closed subset of
$\mathbb{P}\left(\mathrm{H}^{0}\left(\mathbb{P}^{n},
\mathcal{O}_{\mathbb{P}^{n}}\left(n\right)\right)\right)$. So,  a
general hypersurface of degree $n$ in $\mathbb{P}^{n}$ has no bad point.

Let $X$ be a general hypersurface of degree $n$ in
$\mathbb{P}^{n}$ and $H$ be a divisor in $|-K_{X}|$. We claim that
the pair $(X,H)$ is log canonical at every point $P$ on $X$. By a
suitable coordinate change we may assume that the point
$P=[1:0:\cdots:0]$. We may also assume that the hypersurface $X$ is
defined by the equation
$$
z_0^{n-1}z_{n}+z_0^{n-2}q_2(z_1,\cdots,
z_n)+\cdots+z_0q_{n-1}(z_1,\cdots, z_n)+q_n(z_1,\cdots, z_n)=0,
$$
where $q_{i}$ is a homogeneous polynomial of degree $i$. Unless
the hyperplane section $H$ is given by the tangent hyperplane at
the point $P$, the divisor $H$ is smooth at the point $P$, and
hence the pair $(X,H)$ is log canonical at the point $P$. Now we
suppose that $H$ is given by the tangent hyperplane $T$ at the
point $P$. The hyperplane $T$ in $\mathbb{P}^n$ is defined by
$z_n=0$ in our case. Since both $X$ and $T$ are smooth and $H=T\cap
X$, we obtain $c_P(X,H)=c_P(T,H)$ from \cite[Theorem~3.1]{dFEM03}.
Furthermore, $c_P(T,H)=c_0(f)$, where $f=q_2(z_1,\cdots,
z_{n-1},0)+\cdots+q_{n-1}(z_1,\cdots, z_{n-1},0)+q_n(z_1,\cdots, z_{n-1},0)$.
Since $P$ is not a bad point on $X$, the polynomial $q_2(z_1,\cdots,
z_{n-1},0)$
is not zero polynomial. If the rank of the quadratic polynomial
$q_2(z_1,\cdots,
z_{n-1},0)$ is at least $2$, then $c_0(f)=1$ by
\cite[Lemma~8.10 (8.10.3)]{Ko97}. If the rank of the quadratic
polynomial $q_2(z_1,\cdots,
z_{n-1},0)$ is  $1$, we may assume that $q_2(z_1,\cdots,
z_{n-1},0)=z_{n-1}^2$.  Consider the polynomial $f$ with
weights $\mathrm{wt}(z_{1})=\cdots=\mathrm{wt}(z_{n-2})=2$,
$\mathrm{wt}(z_{n-1})=3$. The leading term of $f$ with respect to the weights is $f_w=z_{n-1}^2+q_3(z_1,\cdots,z_{n-2},0,0)$.
Since the polynomial $\widetilde{q_3}=q_3(z_1,\cdots,z_{n-2},0,0)$ is neither zero polynomial nor a cube of a linear polynomial, we obtain  $c_0(\widetilde{q_3})\geq \frac{1}{2}$, and hence $c_0(f_w)=\max\{\frac{1}{2}+c_0(\widetilde{q_3}), 1\}$=1. By \cite[Proposition~2.1]{Ku1}, we have $c_0(f)\geq c_0(f_w)=1$.
Therefore, the pair $(X, H)$ is log canonical
at every point on $X$. Consequently, $\lct_1(X)=1$.
\end{proof}

For smooth quartic threefolds, one can compute all the possible
first global log canonical thresholds by studying normal quartic
surfaces. Here we only list them and the brief idea to compute them as follows:
\begin{proposition}
\label{theorem:auxiliary} Let $X$ be a smooth quartic threefold in
$\mathbb{P}^{4}$.  The first global log canonical threshold $
\mathrm{lct}_{1}\left(X\right)$ is one of the following:
$$
\Big\{
\frac{3}{4},\frac{29}{36},\frac{22}{27},\frac{5}{6},\frac{16}{19},\frac{17}{20},\frac{6}{7},\frac{13}{15},\frac{37}{42},
\frac{7}{8},\frac{8}{9},\frac{9}{10},\frac{23}{26},\frac{11}{12},\frac{12}{13},\frac{13}{14},\frac{14}{15},\frac{15}{16},\frac{31}{34},
\frac{17}{18},\frac{21}{22},\frac{23}{24},\frac{29}{30},\frac{41}{42},1
\Big\}.
$$
Furthermore, for each number $\mu$ in the set above, there is a
smooth quartic threefold $X$ with $\mathrm{lct}_{1}(X)=\mu$.
\end{proposition}

Its proof goes as follows. A divisor $S\in |-K_X|$ is given by the intersection of $X$ and a hyperplane $H$ in $\mathbb{P}^4$.
Because the log canonical threshold $c(X, S)$ is equal to the log canonical threshold $c(H, S)$ (see \cite[Theorem~3.1]{dFEM03}),
the result above can be obtained by investigating log canonical
thresholds of normal quartic surfaces $H$ in $\mathbb{P}^3$. Note that a hyperplane section of a smooth hypersurface in $\mathbb{P}^n, n\geq 4$ is normal and that a normal hypersurface in $\mathbb{P}^{n-1}$ can be attained by a hyperplane section of a smooth hypersurface in $\mathbb{P}^n$ (see \cite{Ish82}).
Let $S$ be a normal surface in $\mathbb{P}^3$ defined by a
homogeneous quartic polynomial $F$. We suppose that $S$ has a
singular point at [0:0:0:1]. We  then consider the log pair
$(\mathbb{C}^3,D)$, where $D$ is the fourth affine piece of $S$
that is defined by
 the polynomial $f(x,y,z)=F(x,y,z,1)$.
Since log canonical thresholds can be computed locally, it is
enough to study the log pair $(\mathbb{C}^3,D)$ instead of
$(X,S)$. For the detail of the computation, see \cite{W10}.

Before we prove Theorem~\ref{theorem:main}, let us explain our
generality condition. Let $X_d$ be a hypersurface of degree $d$ in
$\mathbb{P}^n$, $d\geq n\geq 4$. Let $P$ be an arbitrary point on
$X_d$. By suitable coordinate changes, we assume that
$P=[1:0:\cdots:0]$. Then the hypersurface $X_d$ is defined by
\[z_0^{n-1}q_1(z_1,\cdots, z_n)+z_0^{n-2}q_2(z_1,\cdots,
z_n)+\cdots+z_0q_{n-1}(z_1,\cdots, z_n)+q_d(z_1,\cdots, z_n)=0,\]
 where $q_i$ are
homogeneous polynomials of degrees $i$ in variables $z_1, \cdots,
z_n$.
\begin{definition}\label{definition:regularity}
The hypersurface $X_d$ is said to be $k$-regular at the point $P$,
where $0\leq k\leq d$, if the homogenous polynomials
$$q_1, q_2, \cdots, q_k$$
form a regular sequence in $\mathbb{C}[z_1,\cdots,z_n]$. The
hypersurface $X_d$ is said to be $k$-regular if it is $k$-regular
everywhere.
\end{definition}

\begin{proposition}
\label{proposition:n-regular} A general hypersurface of degree $n$
in $\mathbb{P}^n$ is $(n-1)$-regular.
\end{proposition}
\begin{proof}
See \cite[Proposition~1]{Pu98a}.
\end{proof}

To prove Theorem~\ref{theorem:main} we need a linear system on
$X_d$ that has a big multiplicity at a given point  but a small
base locus. Put
$$f_i(z_0,\cdots,z_n)=\sum_{j=1}^iz_0^{i-j}q_j(z_1,\cdots,z_n)$$
for each $1\leq i\leq d$.
\begin{definition}
The $m$-th hypertangent linear system
$\mathcal{M}$ at the point $P$ is the linear subsystem of
$|-mK_{X_d}|$ consisting of the divisors cut by hypersurfaces
\[\sum_{i=1}^{m}f_i(z_0,\cdots,z_n)p_{m-i}(z_1,\cdots,z_n)=0,\]
where $p_{j}(z_1,\cdots,z_n)$ is a homogeneous polynomial of
degree $j$.
\end{definition}
Note that ${mult}_{P}\left(M\right)\geq m+1 $ for each divisor $M$
in the $m$-th hypertangent linear system on $X_d$.

\begin{lemma}\label{lemma:hypertangent-base-locus}
Suppose that the hypersurface $X_d$ is $(n-1)$-regular at a point
$P$. Then the following hold. \begin{enumerate} \item There are
finitely many lines (possibly none) on $X_d$ passing through the
point $P$. \item The base locus of the $(n-1)$-th hypertangent
linear system $\mathcal{M}$ at the point $P$ consists of lines
passing through the point $P$ on $X_d$.
\end{enumerate}
\end{lemma}
\begin{proof}
There is a one-to-one  correspondence between the set of lines passing through the point $P$ and the zero locus of the polynomials $q_1=\cdots =q_d=0$ in $\mathbb{P}^{n-1}$.
Since the homogeneous  polynomials
$q_1,\cdots, q_{n-1}$ form a regular sequence in
$\mathbb{C}[z_1,\cdots, z_n]$, they defines a finite set  in $\mathbb{P}^{n-1}$. This proves the first assertion.

The base locus of the linear system
$\mathcal{M}$ is defined by the equations
$f_1=\cdots=f_{n-1}=0$. Therefore, it is cut out by the equations $q_1=q_2=\cdots =q_{n-1}=0$.
This shows the second assertion.
\end{proof}
We close this section by the following useful lemma.
\begin{lemma}
\label{lemma:single-point} Let $X$ be a smooth hypersurface of degree $n$ in $\mathbb{P}^n$ and $D$ be an effective $\mathbb{Q}$-divisor numerically equivalent to $-K_X$. For a non-negative number $\lambda \leq 1$, there is a point $P\in X$ such that
$(X,\lambda D)$ is log canonical on $X\setminus P$.
\end{lemma}

\begin{proof}
The log pair $(X,\lambda D)$ is log canonical in the outside of
finitely many points of the smooth hypersurface $X$ (see
\cite[Theorem~2]{Pu95} or \cite[Section~3]{Pu98a}).
Suppose that there are two points at which the log pair  $(X,\lambda D)$ is not log canonical. Then for sufficiently small $\epsilon >0$ the log pair
$(X,(\lambda-\epsilon) D)$ is not log canonical at the two points either. Since the divisor $-(K_X+(\lambda-\epsilon)D)$ is nef and big, it follows from  from  the
connectedness principle of Shokurov (see
\cite[Theorem~17.4]{Koetal92}) that the locus of non-Kawamata log terminal singularities of the log pair $(X,(\lambda-\epsilon) D)$ is connected. This is a contradiction.
\end{proof}

\subsection{General quartic.}
\label{section:global-threshold}

Let $X$ be a smooth quartic hypersurface in $\mathbb{P}^{4}$ such
that the following general conditions hold:
\begin{itemize}
\item the threefold $X$ is $3$-regular;%
\item every line on the hypersurface $X$ has normal bundle $\mathcal{O}_{\mathbb{P}^{1}}(-1)\oplus\mathcal{O}_{\mathbb{P}^{1}}$;%
\item the intersection of $X$ with a two-dimensional linear
subspace of $\mathbb{P}^4$ cannot be a double conic curve.
\end{itemize}

\begin{remark}
\label{remark:Collino} A line on the quartic $X$ has normal bundle
$\mathcal{O}_{\mathbb{P}^{1}}(-1)\oplus\mathcal{O}_{\mathbb{P}^{1}}$
if and only~if no two-dimensional linear subspace of
$\mathbb{P}^{4}$ is tangent to the quartic $X$ along the line (see
\cite[Theorem~1.9]{Col79}).
\end{remark}

\begin{remark}
\label{remark:hyperplanes} It follows from
Proposition~\ref{theorem:auxiliary} and \cite{ChPa02}  that
$\mathrm{lct}_{1}(X)\geq \frac{7}{9}$. To avoid the long proof of
Proposition~\ref{theorem:auxiliary}, we can use instead
Proposition~\ref{proposition:general-1st-lct} by adding extra generality conditions.
\end{remark}

\begin{remark}
\label{remark:weighted-sum} Let $B$ and $B^{\prime}$ be effective
$\mathbb{Q}$-Cartier $\mathbb{Q}$-divisors on a variety $V$. Then
$$
\left(V,\ \alpha B+\left(1-\alpha\right)B^{\prime}\right)
$$
is log canonical if both $(V, B)$ and $(V, B^{\prime})$ are log
canonical, where $0\leq\alpha\leq 1$.
\end{remark}

Let us prove Theorem~\ref{theorem:main} for the case $n=4$. Put
$\lambda=\frac{7}{9}$. Let $D$ be an effective
$\mathbb{Q}$-divisor on $X$ such that $D\equiv -K_{X}$. To prove
Theorem~\ref{theorem:main}, we have to show that $(X,\lambda D)$
is log canonical.

Suppose that $(X,\lambda D)$ is not log canonical. Due to
Remarks~\ref{remark:hyperplanes} and \ref{remark:weighted-sum}, we
may assume that $D=\frac{1}{n}R$ where $R$ is an irreducible
divisor with $R\sim-nK_{X}$ for some natural number $n>1$. By
Lemma~\ref{lemma:single-point}, the log pair $(X, \lambda D)$ is
not log canonical only at a single point $P$.

 The threefold $X$ can
be~given~by
$$
v^{3}x+v^{2}q_{2}\left(x,y,z,u\right)+vq_{3}\left(x,y,z,u\right)+q_{4}\left(x,y,z,u\right)=0\subset\mathbb{P}^{4}\cong\mathrm{Proj}\left(\mathbb{C}\big[x,y,z,u,v\big]\right),
$$
where $q_{i}(x,y,z,u)$ is a homogeneous polynomial of degree $i$.
Furthermore, we may assume that the point $P$ is located at
$[0:0:0:0:1]$.  Let $T$ be the surface on $X$ cut out by $x=0$.

\begin{lemma}
\label{lemma:mult-at-P} The multiplicity of $D$ at the point $P$
is at most $2$.
\end{lemma}

\begin{proof}
The statement immediately follows from the inequalities
$$
4=H\cdot T\cdot D\geq\mathrm{mult}_{P}\left(T\cap D\right)\geq\mathrm{mult}_{P}\left(T\right)\mathrm{mult}_{P}\left(D\right)\geq 2\mathrm{mult}_{P}\left(D\right),%
$$
where $H$ is a general hyperplane section of $X$ passing through
the point $P$.
\end{proof}

Let $\pi\colon U\to X$ be the blow up at the point $P$ with the exceptional divisor $E$. Then
$$
\bar{D}\equiv\pi^{*}\left(D\right)-\mathrm{mult}_{P}\left(D\right)E,
$$
where $\bar{D}$ is the proper transform of the divisor $D$ via the
morphism $\pi$.

It follows from \cite[Corollary~3.5]{Co00} or
\cite[Proposition~3]{Pu04d} that there is a line $L\subset E$ such
that
$$
\mathrm{mult}_{P}\left(D\right)+\mathrm{mult}_{L}\left(\bar{D}\right)>\frac{2}{\lambda}.
$$
Recall that $E$ is isomorphic to $\mathbb{P}^2$.

Let $\mathcal{L}$ be the linear system of hyperplane sections of
$X$ such that
$$
S\in\mathcal{L}\iff\text{either}\ L\subset\bar{S}\ \text{or}\ S=T,%
$$
where $\bar{S}$ is the proper transform of $S$ via the birational
morphism $\pi$. There is a two-dimensional linear subspace
$\Pi\subset\mathbb{P}^{4}$ such that the base locus of
$\mathcal{L}$ consists of the intersection $\Pi\cap X$.

Let $S$ be a general surface in $\mathcal{L}$. Then $S$ is a
smooth K3 surface.  Put
$$
T_S=T\big\vert_{S}=\sum_{i=1}^{r}Z_{i},
$$
where each $Z_{i}$ is an irreducible curve. The generality conditions
imply that the curve $T_{S}$ is reduced (see
Remark~\ref{remark:Collino}). Then
$\sum_{i=1}^{r}\mathrm{deg}(Z_{i})=4$. It follows that
$$
\mathrm{mult}_{P}\left(S\cap D\right)
\geq\mathrm{mult}_{P}\left(S\right)\mathrm{mult}_{P}\left(D\right)+\mathrm{mult}_{L}\left(\bar{S}\cap\bar{D}\right)
\geq\mathrm{mult}_{P}\left(D\right)+\mathrm{mult}_{L}\left(\bar{D}\right)>\frac{2}{\lambda}.%
$$

Put $$
 D_S=D\big\vert_{S}=\sum_{i=1}^{r}m_{i}Z_{i}+\Delta,
$$
where $m_{i}$ is a non-negative rational number and $\Delta$ is an
effective one-cycle on $S$ whose support does not contain the
curves $Z_{1},\ldots,Z_{r}$. Then
$$
\sum_{i=1}^{r}m_{i}\mathrm{mult}_{P}\left(Z_{i}\right)+\mathrm{mult}_{P}\left(\Delta\right)=\mathrm{mult}_{P}\left(
D_S\right)>\frac{2}{\lambda},
$$
and the support of the cycle $\Delta$ does not contain any
component of the cycle $ T_S$. We have
$$
4= T_S\cdot D_S=\sum_{i=1}^{r}m_{i}\mathrm{deg}\left(Z_{i}\right)+
T_S\cdot\Delta
\geq\sum_{i=1}^{r}m_{i}\mathrm{deg}\left(Z_{i}\right)+\mult_P(T_S)\left(\frac{2}{\lambda}-\sum_{i=1}^{r}m_{i}\mathrm{mult}_{P}\left(Z_{i}\right)\right).%
$$

\begin{remark}
\label{remark:multiplicities} The equality
$m_{i}=\mathrm{mult}_{Z_{i}}(D)$ holds for every $i$ because
$X\vert_{\Pi}$ is reduced.
\end{remark}

It follows from the $3$-regularity of $X$ that
$\mathrm{mult}_{P}(T_S)\leq 3$.

\begin{lemma}
\label{lemma:mult-3} Suppose that $\mathrm{mult}_{P}(T_S)=3$. Then
$$
16>\frac{12}{\lambda}+\mathrm{deg}\left(Z_{k}\right)m_{k},
$$
for $Z_k$ that is not a line passing through the point $P$.
\end{lemma}

\begin{proof}
Let $\bar{T}$ be the proper transform of the surface $T$ via the
birational morphism $\pi$. Then
$$
3=\mathrm{mult}_{P}(T_S)=\mathrm{mult}_{P}\left(T\cap S\right)=\mathrm{mult}_{P}\left(T\right)\mathrm{mult}_{P}\left(S\right)+\mathrm{mult}_{L}\left(\bar{T}\cap\bar{S}\right).%
$$
Hence, we see that $L\subset\bar{T}$. Since
$\mathrm{mult}_{P}(D)>\frac{1}{\lambda}$ and
$\mathrm{mult}_{P}(T)=2$, it follows that
$$
\mathrm{mult}_{P}\left(T\cap D\right)\geq\mathrm{mult}_{P}\left(T\right)\mathrm{mult}_{P}\left(D\right)+\mathrm{mult}_{L}\left(\bar{T}\cap\bar{D}\right)\geq 2\mathrm{mult}_{P}\left(D\right)+\mathrm{mult}_{L}\left(\bar{D}\right)>\frac{3}{\lambda}.%
$$

Let $L_{1},\ldots,L_{m}$ be all the lines on $X$ that pass through the
point $P$. Put
$$
T\cap D=\sum_{i=1}^{m}\epsilon_{i}L_{i}+\bar{m}_{k}Z_{k}+\Upsilon,%
$$
where $\epsilon_{i}$ and $\bar{m}_{k}$ are  non-negative rational
numbers, and $\Upsilon$ is an effective one-cycle on $X$ whose
support does not contain the lines $L_{1},\ldots,L_{m}$. Then
$\bar{m}_{k}\geq m_{k}$ by Remark~\ref{remark:multiplicities}.

Taking the intersection with a general hyperplane section of $X$,
we see that
$$
4\geq\sum_{i=1}^{r}\epsilon_{i}+\bar{m}_{k}\mathrm{deg}\left(Z_{k}\right),%
$$
but
$\bar{m}_{k}\mathrm{mult}_{P}(Z_{k})+\mathrm{mult}_{P}(\Upsilon)>\frac{3}{\lambda}-\sum_{i=1}^{r}\epsilon_{i}$.

Take a general member $M$ in  the third hypertangent linear system
$\mathcal{M}$ at the point $P$. Note that the base locus of
$\mathcal{M}$ consists of the lines $L_{1},\ldots,L_{m}$ by
Lemma~\ref{lemma:hypertangent-base-locus}. Hence, we have
\[\begin{split}
12 &=M\cdot T\cdot D\geq
3\sum_{i=1}^{r}\epsilon_{i}+M\cdot\left(\bar{m}_{k}Z_{k}+\Upsilon\right)
\\
&>3\sum_{i=1}^{r}\epsilon_{i}+4\left(\frac{3}{\lambda}-\sum_{i=1}^{r}\epsilon_{i}\right)
=\frac{12}{\lambda}-\sum_{i=1}^{r}\epsilon_{i}.%
\end{split}\]
This implies $16>12/\lambda+\mathrm{deg}(Z_{k})m_{k}$ since
$4\geq\sum_{i=1}^{r}\epsilon_{i}+\bar{m}_{k}\mathrm{deg}(Z_{k})$
and
 $\bar{m}_{k}\geq m_{k}$.
\end{proof}

From now on, in order to describe the reduced curve $T_S$, we will
use the following notations:
\begin{itemize}
\item $C$ : an irreducible cubic not passing through the point
$P$.

\item $\widetilde{C}$ : an irreducible cubic that is smooth at the
point $P$.

\item  $\widehat{C}$ : an irreducible cubic that is singular at
the point $P$.
\end{itemize}

For $i=1,2$
\begin{itemize}

\item $Q_i$ : an irreducible quadric not passing through the point
$P$.

\item $\widetilde{Q_i}$ : an irreducible quadric passing through
the point $P$.
\end{itemize}

For $i=1,2,3,4$
\begin{itemize}

\item $L_i$ : a line not passing through the point $P$.

\item $\widetilde{L_i}$ : a line passing through  the point $P$.
\end{itemize}

Then, the following are all the possible configuration of $T_S$.
In each case, we derive a contradictory inequality from our
assumptions so that the log pair $(X, \lambda D)$ should be log
canonical.
To obtain a contradictory inequality for each case, we start from the inequality
\[
\begin{split}
4=T_S\cdot D_S =\sum m_iZ_i\cdot T_S+T_S\cdot\Delta &\geq \sum
m_i\deg(Z_i)+\mathrm{mult}_P(T_S)\mathrm{mult}_P(\Delta) \\ &>
\sum m_i\deg(Z_i)+\mathrm{mult}_P(T_S)\left(\frac{2}{\lambda}-\sum
m_i\mathrm{mult}_P(Z_i)\right),\\ \end{split}\] and then we show
that the number
\[A:=\sum m_i\deg(Z_i)+\mathrm{mult}_P(T_S)\left(\frac{2}{\lambda}-\sum m_i\mathrm{mult}_P(Z_i)\right)\]
is greater than $4$.
\bigskip

\textbf{CASE A.} The curve $T_S$ is an irreducible quartic curve.
\begin{enumerate}

\item $\mathrm{mult}_P(T_S)=2$.

$ D_S= mT_S+\Delta$.

A contradictory inequality:
\[ A=4m +2\left(\frac{2}{\lambda}-2m\right)=\frac{4}{\lambda}>4.\]

\item  $\mathrm{mult}_P(T_S)=3$.

$ D_S= mT_S+\Delta$.

An auxiliary inequality:

$$16>\frac{12}{\lambda}+4m \ \ \textnormal{ by Lemma~\ref{lemma:mult-3}}.$$

A contradictory inequality:
\[
A=4m+3\left(\frac{2}{\lambda}-3m\right)=\frac{6}{\lambda}-5m> \frac{6}{\lambda}-5\left(4-\frac{3}{\lambda}\right)=\frac{21}{\lambda}-20>4.
\]

\end{enumerate}

\textbf{CASE B.} The curve $T_S$ is reducible and contains no line
passing through the point $P$.
\begin{enumerate}

 \item $T_S= \widehat{C}+L_1$.

$ D_S= m\widehat{C}+m_1L_1+\Delta$.

An auxiliary inequality :
\[1=L_1\cdot  D_S\geq 3m-2m_1.\]

A contradictory inequality:
\[A = 3m +m_1+2\left(\frac{2}{\lambda}-2m\right)=\frac{4}{\lambda}+m_1-m  \geq \frac{4}{\lambda}+m_1-\frac{1+2m_1}{3} >4.\]

 \item $T_S= \widetilde{Q_1}+\widetilde{Q_2}$.

$ D_S= m_1\widetilde{Q_1}+m_2\widetilde{Q_2}+\Delta$.

A contradictory inequality:
\[A= 2m_1 +2m_2+2\left(\frac{2}{\lambda}-m_1-m_2\right)=\frac{4}{\lambda}>4.\]

 \end{enumerate}

\textbf{CASE C.}  The curve $T_S$ contains a unique line passing
through the point $P$.
\begin{enumerate}

 \item $T_S= \widetilde{Q_1}+\widetilde{L_1}+L_2$.

$ D_S=m\widetilde{Q_1}+m_1\widetilde{L_1}+m_2L_2+\Delta$.

Auxiliary inequalities:
\[\left. \begin{array}{l} 2=Q_1\cdot  D_S\geq -2m+2m_1+2m_2 \\
1=L_2\cdot  D_S\geq 2m+m_1-2m_2 \end{array}\right\}\Rightarrow
1\geq m_1
\]

A contradictory inequality:
\[A=2m +m_1+m_2+2\left(\frac{2}{\lambda}-m-m_1\right) =\frac{4}{\lambda}+m_2-m_1\geq \frac{4}{\lambda}+m_2-1 >4.\]

 \item $T_S= \widetilde{C}+\widetilde{L_1}$.

$ D_S= m\widetilde{C}+m_1\widetilde{L_1}+\Delta$.

An auxiliary inequality:
\[3=\widetilde{C}\cdot  D_S\geq 3m_1.\]

A contradictory inequality:
\[A=3m +m_1+2\left(\frac{2}{\lambda}-m-m_1\right) =\frac{4}{\lambda}+m-m_1\geq \frac{4}{\lambda}+m-1 >4.\]

 \item $T_S= \widehat{C}+\widetilde{L_1}$.

$D_S=m\widehat{C}+m_1\widetilde{L_1}+\Delta$.

Auxiliary inequalities:
\[\begin{array}{l} 3=\widehat{C}\cdot D_S\geq 3m_1 \\
16> \frac{12}{\lambda}+3m   \ \ \mbox{ by
Lemma~\ref{lemma:mult-3}}\end{array}\]

A contradictory inequality:
\[\begin{split} A &= \frac{6}{\lambda}-3m-2m_1 \geq \frac{6}{\lambda}-2-3m  > \frac{6}{\lambda}-2-\left(16-\frac{12}{\lambda}\right)=\frac{18}{\lambda}-18 >4.\end{split}\]

 \end{enumerate}

\textbf{CASE D.} The curve $T_S$ contains two lines passing
through the point $P$.
\begin{enumerate}
 \item $T_S=\widetilde{Q_1}+\widetilde{L_1}+\widetilde{L_2}$.

$D_S=m\widetilde{Q_1}+m_1\widetilde{L_1}+m_2\widetilde{L_2}+\Delta$,

Auxiliary inequalities:
\[\begin{array}{l} 2=\widetilde{Q_1}\cdot D_S\geq -2m+2m_1+2m_2 \\
16> \frac{12}{\lambda}+2m   \ \ \mbox{ by
Lemma~\ref{lemma:mult-3}.}\end{array}\]

A contradictory inequality:
\[\begin{split}
A=\frac{6}{\lambda}-m-2m_1-2m_2  \geq \frac{6}{\lambda}-m-2\left(1+m\right)  > \frac{6}{\lambda}-2-3\left(8-\frac{6}{\lambda}\right)=
\frac{24}{\lambda}-26> 4. \end{split}
\]

 \item $T_S= \widetilde{L_1}+\widetilde{L_2}+L_3+L_4$.

$ D_S=m_1\widetilde{L_1}+m_2\widetilde{L_2}+m_3L_3+m_4L_4
+\Delta$, where we may assume that $m_3\geq m_4$.

An auxiliary inequality:
\[1=L_4\cdot  D_S\geq m_1+m_2+m_3-2m_4\]

A contradictory inequality:
\[\begin{split} A =\frac{4}{\lambda}+m_3+m_4-\left(m_1+m_2\right)  &\geq \frac{4}{\lambda}+m_3+m_4-m_2-\left(1-m_2-m_3+2m_4\right) \\ &\geq \frac{4}{\lambda}+2m_3-m_4-1 >4. \end{split}\]

 \item $T_S=Q_1+\widetilde{L_1}+\widetilde{L_2}$.

$ D_S=mQ_1+m_1\widetilde{L_1}+m_2\widetilde{L_2}+\Delta$.

An auxiliary inequality:
\[2=Q_1\cdot  D_S\geq -2m+2m_1+2m_2\ \Rightarrow 1+m\geq m_1+m_2. \]

A contradictory inequality:
\[\begin{split} A=\frac{4}{\lambda}+2m-m_1-m_2\geq \frac{4}{\lambda}+m-1 >4.
\end{split}\]

\end{enumerate}
\textbf{CASE E.} The curve $T_S$ contains three lines passing
through the point $P$.

\begin{enumerate}
 \item $T_S= \widetilde{L_1}+\widetilde{L_2}+\widetilde{L_3}+L_4$.

$D_S=m_1\widetilde{L_1}+m_2\widetilde{L_2}+m_3\widetilde{L_3}+mL_4+\Delta$.

Auxiliary inequalities:
\[\begin{array}{l} 1=L_4\cdot D_S\geq -2m+m_1+m_2+m_3 \\
16> \frac{12}{\lambda}+m   \ \ \mbox{ by
Lemma~\ref{lemma:mult-3}.}\end{array}\]

A contradictory inequality:
\[\begin{split}
A=\frac{6}{\lambda}+m-2\left(m_1+m_2+m_3\right) \geq \frac{6}{\lambda}+m-2\left(1+2m\right)> \frac{6}{\lambda} -2 -3\left(16-\frac{12}{\lambda}\right)=\frac{42}{\lambda}-50 =4. \end{split}
\]

\end{enumerate}

Therefore, Theorem~\ref{theorem:main} for $n=4$ has been proved.

\subsection{General quintic.}\label{subsection:general-quintic}
In this section, we prove Theorem~\ref{theorem:main} for $n=5$.

Let $X$ be a quintic hypersurface in $\mathbb{P}^{5}$ such that
the following generality conditions hold:
\begin{itemize}
\item[G\,1.]  The hypersurface $X$ is $4$-regular;
\item [G\,2.] For every $3$-dimensional linear space $\Pi$ in $\mathbb{P}^5$, the intersection $X\cap\Pi$ is irreducible and reduced;
\item[G\,3.] For each point $P\in X$ and each $3$-dimensional linear space $\Pi$ contained in the tangent hyperplane at $P$ and containing the point $P$, if the surface $Z:=X\cap\Pi$ has multiplicity two at the point $P$, then it satisfies the following:
\begin{itemize}
\item[G\,3.0.] The surface $Z$
\begin{itemize}
\item[G\,3.0.1.] cannot be singular along a line passing through the point $P$ ;
\item[G\,3.0.2.] cannot contain four lines passing through the point $P$.

\end{itemize}
\item[G\,3.1.] If $Z$ contains only one line $L$ passing through the point $P$,
\begin{itemize}
\item[G\,3.1.1.] then the line $L$ meets its residual curve by a general hyperplane section in $\Pi$ either at at least one smooth point or at at least two ordinary double points.
 
\end{itemize}
\item[G\,3.2.] If $Z$ contains only two lines,  $L_1$ and $L_2$, passing through the point $P$, then
\begin{itemize}
\item[G\,3.2.1.]  it has at most four singular points on  $L_1\cup L_2$;
\item[G\,3.2.2.]  if it has four singular points  on $L_i$, then
all of them are ordinary double points; \item[G\,3.2.3.]  if it
has three singular points  on $L_i$, then two  of them are
ordinary double points; \item[G\,3.2.4.] if it has exactly
three singular points on the line $L_i$, then the line $L_i$ meets
 its residual curve by a general hyperplane section in $\Pi$ at  one smooth point; 
\item[G\,3.2.5.] if it has exactly two  singular points on
the line $L_i$, then either $P$ is a non-ordinary double point and
the line $L_i$ meets its residual curve by a general hyperplane
section in $\Pi$ at two smooth points, or the point $P$ is an
ordinary double point and the line $L_i$ meets its residual curve
by a general hyperplane section in $\Pi$ at at least one smooth
point. \item[G\,3.2.6.] if it has no singular point other
than $P$ on the line $L_i$, then the line $L_i$ meets
 its residual curve by a general hyperplane section in $\Pi$ at  at least two smooth points. 

\end{itemize}
\item[G\,3.3.] If $Z$ contains three lines, $L_1$, $L_2$ and $L_3$, passing through the point $P$,
\begin{itemize}
\item[G\,3.3.1.]if the three lines are coplanar, then it is
smooth on $(L_1\cup L_2\cup L_3)\setminus \{P\}$ and each line
$L_i$ meets its residual curve by a general hyperplane section in
$\Pi$  at  four points; \item[G\,3.3.2.] if the three lines
are not coplanar, then either $P$ is a non-ordinary double point,
the surface $Z$ is smooth at every point of $(L_1\cup L_2\cup
L_3)\setminus \{P\}$ and each line $L_i$ meets its residual curve
by a general hyperplane section in $\Pi$  at  four points, or $P$
is an ordinary double point and each line $L_i$ meets its residual
curve by a general hyperplane section in $\Pi$ at two smooth
points.
\end{itemize}
\end{itemize}
\end{itemize}

\begin{lemma}\label{lemma:G2}
A general quintic hypersurface $X$ in $\mathbb{P}^5$ satisfies the condition G\,2.
\end{lemma}
\begin{proof}
This follows directly from \cite[Theorem~5.1]{CChG08}. \end{proof}

\begin{lemma}\label{lemma:G3}
A general quintic hypersurface $X$ in $\mathbb{P}^5$ satisfies the condition G\,3.
\end{lemma}
\begin{proof}
See Appendix.
\end{proof}

Put
$\lambda=\frac{5}{6}$. Let $D$ be an effective
$\mathbb{Q}$-divisor on $X$ such that $D\equiv -K_{X}$. We claim that the log pair $(X,\lambda D)$
is log canonical.

Suppose that the log pair $(X,\lambda D)$ is not log canonical.  As in the case of quartic threefolds, we may assume that $D=\frac{1}{n}R$
where $R$ is an irreducible divisor with
$R\sim-nK_{X}$ for some natural number $n$.
Furthermore, the following lemma enables us to assume $n>1$.
We may use Proposition~\ref{proposition:general-1st-lct} with extra generality conditions in order to assume $n>1$ without the aid of the lemma.

\begin{lemma}
If a quintic hypersurface $Y$ in $\mathbb{P}^5$ is $4$-regular, then
$\lct_1(Y)=1$.
\end{lemma}

\begin{proof}
The main idea of the proof is the same as that of
Proposition~\ref{theorem:auxiliary}. It is enough to prove
$c_0(f)=1$ for a quintic polynomial $f(x,y,z,u)\in
\mathbb{C}[x,y,z,u]$ obtained from the quintic polynomial defining
the quintic $Y$. Using the $4$-regular condition we can derive
enough monomials from the polynomial $f$ to have $c_0(f)=1$. We
omit the detailed computation. For the detail, see \cite{W10}.
\end{proof}

 It
follows from Lemma~\ref{lemma:single-point} that there is a point
$P\in X$ such that the log pair $(X,\lambda D)$ is log canonical
on $X\setminus P$. Therefore, the log pair $(X, \lambda D)$ is not
log canonical only at the point $P$.

By suitable coordinate changes, we may assume that $P=[0:0:0:0:0:1]$ and that the
fourfold $X$ is given by an equation
$$
w^{4}x+\sum_{i=2}^{5}w^{5-i}q_{i}\left(x,y,z,u,v\right)=0\subset\mathbb{P}^{5}\cong\mathrm{Proj}\left(\mathbb{C}\big[x,y,z,u,v,w\big]\right),
$$
where
$q_{i}(x,y,z,u,v)$ is a homogeneous polynomial of degree $i$.   Let
$T$ be the threefold  on $X$ cut by $x=0$.

Let $\pi\colon U\to X$ be the blow up at the point $P$ with the exceptional divisor $E$. Then
$$
\bar{D}\equiv\pi^{*}\left(D\right)-\mathrm{mult}_{P}\left(D\right)E,
$$
where $\bar{D}$ is the proper transform of the divisor $D$ via the
morphism $\pi$. Note that
$\mathrm{mult}_{P}(D)>\frac{1}{\lambda}$.  It follows from
\cite[Proposition~3]{Pu04d} that either
$\mathrm{mult}_{P}(D)>\frac{2}{\lambda}$ or there is a plane
$\Omega\subset E\cong\mathbb{P}^{3}$ such that
$$
\mathrm{mult}_{P}\left(D\right)+\mathrm{mult}_{\Omega}\left(\bar{D}\right)>\frac{2}{\lambda}.
$$

In the case when $\mathrm{mult}_{P}(D)>\frac{2}{\lambda}$, let
$\mathcal{L}$ be a sufficiently general pencil of hyperplane
sections of $X$ that pass through the point $P$.
In the case when
$\mathrm{mult}_{P}(D)\leq \frac{2}{\lambda}$, let $\mathcal{L}$ be
the pencil of hyperplane sections of $X$ such that
$$
S\in\mathcal{L}\iff\text{either}\ \Omega\subset\bar{S}\ \text{or}\ S=T,%
$$
where $\bar{S}$ is the proper transform of $S$ via the birational
morphism $\pi$. In both the cases, there is a three-dimensional linear subspace
$\Pi\subset\mathbb{P}^{5}$ such that the base locus of
$\mathcal{L}$ consists of the intersection $\Pi\cap X$.

Let $S$ be a general threefold in $\mathcal{L}$. Then $S\ne T$ and
$\mathrm{mult}_{P}(S\cap D)>\frac{2}{\lambda}$.

Put $Z=X\vert_{\Pi}$. The surface $Z$ is reduced and irreducible
because $X$  contains neither quadric surfaces nor planes by our
initial assumption. The $4$-regularity of $X$ implies that
$\mathrm{mult}_{P}(Z)\leq 3$.

\begin{lemma}\label{lemma:mult 3}
\label{lemma:mult-2} The multiplicity of $Z$ at the point $P$ is $3$.
\end{lemma}

\begin{proof}
Suppose that $\mathrm{mult}_{P}(Z)\leq 2$.
Let $\mathcal{M}$ be the 4th hypertangent linear system at the point $P$ and let $M$ be a general member in $\mathcal{M}$.
The base locus of $\mathcal{M}$ consists of finitely many
lines  on $X$ that pass through the point $P$.

Put $D\cap
S=mZ+\Upsilon$, where $m$ is a non-negative rational number and
$\Upsilon$ is a $2$-cycle whose support does not contain the
surface $Z$. Then
$
\mathrm{mult}_{P}\left(\Upsilon\right)>\frac{2}{\lambda}-2m
$
but $T$ does not contain components of $\Upsilon$. We therefore have
$$\mathrm{mult}_{P}(T\cap\Upsilon)>\frac{4}{\lambda}-4m.$$

We then consider the one cycle  $T\cap\Upsilon$. We may write
\[T\cap\Upsilon=\sum_{i=1}^k\alpha_iL_i +\Delta,\]
where $L_i$ is a line contained in $Z$ and passing through the point $P$ and the support of $\Delta$ contain none of the lines
$L_i$'s.
We have
$$
M\cdot \Delta=M\cdot\left(T\cdot D\cdot S-mT\cdot Z-\sum_{i=1}^k\alpha_iL_i\right)=20-20m-4\sum_{1=1}^k\alpha_i%
$$
and
$$
M\cdot \Delta\geq \mathrm{mult}_{P}\left(M\right)\mathrm{mult}_{P}\left(\Delta\right)>\frac{20}{\lambda}-20m-5\sum_{i=1}^k\alpha_i,%
$$
and hence
\[4=\frac{20}{\lambda}-20<\sum_{i=1}^{k}\alpha_i.\]

On the other hand, using our generality condition G\,3, we obtain the opposite inequality $\sum_{i=1}^{k}\alpha_i\leq 4$ case by case as follows, so that we could conclude that $\mult_P(Z)$=3.

By our generality condition, we have $k\leq 3$. Note that we may regard $T\cap\Upsilon$
as a divisor in $|\mathcal{O}_{Z}(1-m)|$  on the quintic surface $Z\subset\Pi\cong\mathbb{P}^3$ since $T\cap S= Z$. For each line $L_j$  we consider the hyperplane section $A_j$ of $Z$ by a general hyperplane in $\Pi$ passing through the line $L_j$.  The divisor $A_j$ on the surface $Z$ consists of the line $L_j$ and the residual curve $C_j$. On the surface $Z$, we have
\begin{equation}\label{inequlality:alpha}
 \sum_{i=1}^k \alpha_iC_j\cdot L_i\leq C_j\cdot(\sum_{i=1}^k\alpha_iL_i +\Delta)=4(1-m)\leq 4.
\end{equation}

On the surface $Z$, the local intersection number of $C_i$ and $L_j$ at an ordinary double point of $Z$ is well-defined and it is at least $\frac{1}{2}$ if these two curves intersect there. The local intersection number of $C_i$ and $L_j$ at a smooth point of $Z$ is at least $1$ if  these two curves intersect there.

\textbf{CASE $k=1$.}

G\,3.1.1  implies $1\leq C_1\cdot L_1$. Then the inequality~(\ref{inequlality:alpha}) implies $\alpha_1\leq 4$.

\medskip

\textbf{CASE $k=2$.}

First we suppose that neither $L_1$ nor $L_2$ contains exactly two singular points of $Z$.
Then it follows from the conditions G\,3.2.1,~3.2.2,~3.2.3,~3.2.4,~3.2.6 that  $2\leq C_j\cdot L_j$.  This implies that $\alpha_j\leq 2$, and hence $\alpha_1+\alpha_2\leq 4$.

Now we suppose that $L_1$ contains exactly two singular points of $Z$. One of them are the point $P$.

Suppose that $P$ is a non-ordinary double point. Then
$$
2\alpha_1\leq C_1\cdot (\alpha_1L_1+\alpha_2L_2)\leq 4
$$
by G\,3.2.5. Thus, we have $\alpha_1\leq 2$. On the other hand, it
follows from G\,3.2.3, G\,3.2.4, G\,3.2.5 and G\,3.2.6 that
$$
2\alpha_2\leq C_2\cdot (\alpha_1L_1+\alpha_2L_2)\leq 4,
$$
which implies that $\alpha_2\leq 2$. Then $\alpha_1+\alpha_2\leq
4$.

Suppose now that the point $P$ is an ordinary double point. We
obtain from G\,3.2.5 that
\[\frac{3}{2}\alpha_1+\frac{1}{2}\alpha_2\leq C_1\cdot (\alpha_1L_1+\alpha_2L_2)\leq 4.\]
On the other hand, regardless of the number of the singular points on $L_2$, we  see
\[\frac{1}{2}\alpha_1+\frac{3}{2}\alpha_2\leq C_2\cdot (\alpha_1L_1+\alpha_2L_2)\leq 4\]
by G\,3.2.3, G\,3.2.4, G\,3.2.5 and G\,3.2.6, since the point $P$
is  an ordinary double point. These imply that
$\alpha_1+\alpha_2\leq 4$.

\medskip

\textbf{CASE $k=3$.}

Suppose that the three lines are coplanar. Then G\,3.3.1 shows that for each $j=1$, $2$, $3$ we have $3\leq C_j\cdot L_j$, and hence $\alpha_j\leq\frac{4}{3}$. Therefore, we obtain $\alpha_1+\alpha_2+\alpha_3\leq 4$.

Suppose that the three lines are not coplanar. We have to consider
two cases: when the point $P$ is an ordinary double point, and
when the point $P$ is not an ordinary double point.

Suppose that $P$ is not an ordinary double point. Then G\,3.3.2
shows that for each $j=1$, $2$, $3$ we have $3\leq C_j\cdot L_j$,
and hence $\alpha_j\leq\frac{4}{3}$. Therefore, we obtain
$\alpha_1+\alpha_2+\alpha_3\leq 4$.

Suppose that $P$ is an ordinary double point. Then G\,3.3.2 shows
that for each $i$ and $j$,
$$
C_j\cdot L_i \geq \frac{1}{2}+2\delta_{ij},$$ where $\delta_{ij}$
is the Kronecker-delta function, i.e., $\delta_{ij}=1$ if $i=j$;
$\delta_{ij}=0$ if $i\ne j$. Then the
inequality~(\ref{inequlality:alpha}) implies  that for each $1\leq
j\leq 3$, we have $$\frac{1}{2}(\alpha_1+\alpha_2
+\alpha_3)+2\alpha_j\leq  \sum_{i=1}^3 \alpha_iC_j\cdot L_i\leq
4,$$ and hence
$$3(\alpha_1+\alpha_2+\alpha_3)\leq\frac{7}{2}(\alpha_1+\alpha_2+\alpha_3)\leq \sum_{j=1}^{3}\sum_{i=1}^3\alpha_iC_j\cdot L_i\leq 12,$$
which implies that $\alpha_1+\alpha_2\leq 4$. This completes the
proof.
\end{proof}

Let $\bar{T}$ be the proper transform of  $T$ via the
birational morphism $\pi$. Then
$$
3=\mathrm{mult}_{P}(Z)=\mathrm{mult}_{P}\left(T\cap
S\right)=\mathrm{mult}_{P}\left(T\right)\mathrm{mult}_{P}\left(S\right)+\mathrm{mult}_{\Omega}\left(\bar{T}\cap\bar{S}\right),%
$$
which implies that $\Omega\subset\bar{T}$. Since
$\mathrm{mult}_{P}(D)>\frac{1}{\lambda}$ and $\mathrm{mult}_{P}(T)=2$, it
follows that
$$
\mathrm{mult}_{P}\left(T\cap
D\right)\geq\mathrm{mult}_{P}\left(T\right)\mathrm{mult}_{P}\left(D\right)+\mathrm{mult}_{\Omega}\left(\bar{T}\cap\bar{D}\right)\geq
2\mathrm{mult}_{P}\left(D\right)+\mathrm{mult}_{\Omega}\left(\bar{D}\right)>\frac{3}{\lambda}.%
$$

Now we restrict everything to a general hyperplane section of the
fourfold $X$.
Let $H$ be a general hyperplane in $\mathbb{P}^5$ passing through the point $P$.
Put \[\tilde{X}=H\cap X, \ \ \tilde{T}=H\cap T, \ \ \tilde{S}=H\cap S, \ \ \tilde{D}=H\cap D, \ \ \tilde{Z}=H\cap Z, \ \ \tilde{\Upsilon}=H\cap \Upsilon.\]
Let $\tilde{P}=[0:0:0:0:1]$.
The threefold $\tilde{X}$ is $4$-regular at the point $\tilde{P}$. The divisor $\tilde{D}$ is equivalent to $\mathcal{O}_{\mathbb{P}^{4}}(1)\vert_{\tilde{X}}$.

We have
$$
\mathrm{mult}_{\tilde{P}}\left(\tilde{T}\right)=2,\ \
\mathrm{mult}_{\tilde{P}}(\tilde{Z})=3,\ \
\mathrm{mult}_{\tilde{P}}\left(\tilde{T}\cap\tilde{D}\right)>\frac{3}{\lambda},
\ \ \mathrm{mult}_{\tilde{P}}\left(\tilde{S}\cap \tilde{D}\right)>\frac{2}{\lambda}. %
$$

The intersection $\tilde{T}\cap \tilde{S}$ consists of the irreducible reduced
curve $\tilde{Z}$.  Put
$$
\tilde{T}\cap \tilde{D}=\bar{m}\tilde{Z}+\Delta,%
$$
where $\bar{m}$ is a non-negative rational number and $\Delta$
is an effective one-cycle on $\tilde{X}$ whose support does not contain
the curve $\tilde{Z}$. Then,
$\mathrm{mult}_{\tilde{P}}(\Delta)>\frac{3}{\lambda}-3\bar{m}$.

Let $\mathcal{N}$ be the third hypertangent linear system at the point $\tilde{P}$.  Lemma~\ref{lemma:hypertangent-base-locus} shows that the base locus of
$\mathcal{N}$ does not contain any curves because the threefold
$\tilde{X}$ contains no lines passing through the point $\tilde{P}$. Hence, for a general member $N$ in $\mathcal{N}$ we
have
$$
    15=N\cdot \tilde{T}\cdot \tilde{D}\geq
15\bar{m}+N\cdot\Delta>15\bar{m}+4\left(\frac{3}{\lambda}-3\bar{m}\right),%
$$
which implies $15>\frac{12}{\lambda}+3\bar{m}$. Since $
\tilde{D}\cap \tilde{S}=m\tilde{Z}+\tilde{\Upsilon}$ and
$\mathrm{mult}_{\tilde{P}}(\tilde{\Upsilon})>\frac{2}{\lambda}-3m$,
on the surface $\tilde{S}$ we have
$$
5-5m=\tilde{Z}\cap
\tilde{\Upsilon}>\mathrm{mult}_{\tilde{P}}\left(\tilde{Z}\right)\mathrm{mult}_{\tilde{P}}\left(\tilde{\Upsilon}\right)>3\left(\frac{2}{\lambda}-3m\right).
$$
Thus, we see
that $4m>\frac{6}{\lambda}-5$.

The curve $\tilde{Z}$ is reduced and $\tilde{S}$ is a sufficiently general
    hyperplane section of $\tilde{X}$ that contains the curve $\tilde{Z}$. Thus, we
have
$$
m=\mathrm{mult}_{\tilde{Z}}(\tilde{D})\leq\mathrm{mult}_{\tilde{Z}}(\tilde{T}\cap  \tilde{D})=\bar{m},%
$$
which implies
$$15\geq
\frac{12}{\lambda}+3m>\frac{12}{\lambda}+\frac{3}{4}\left(\frac{6}{\lambda}-5\right).$$ It contradicts $\lambda=\frac{5}{6}$.

The obtained contradiction completes the proof of
Theorem~\ref{theorem:main}.

\section{Log canonical threshold on a double space}
\subsection{Generalized global log canonical threshold.}
The Picard group of a smooth Fano hypersurface of degree $n\geq 4$
of $\mathbb{P}^{n}$ is generated by an anticanonical divisor.
Therefore, it is natural that we consider only plurianticanonical
divisors when we define its global  log canonical threshold.
However, in other varieties, it may not be enough. Therefore, we
generalizes the global log canonical threshold as follows:
\begin{definition}
Let $X$ be a $\mathbb{Q}$-factorial variety with at worst log
canonical singularities. For an integral  divisor $D$ on the
variety $X$ and a natural number $m>0$, we define the $m$-th
global log canonical threshold of the  divisor $D$ by the number
$$
\mathrm{lct}_{m}\left(X,D\right)
=\mathrm{inf}\ \left\{c\left(X,\frac{1}{m}H\right)\ \Big|\ H\in \big|mD\big|\right\},%
$$
where the number $\mathrm{lct}_{m}(X, D)$ is defined to be
$\infty$ if the linear system $|mD|$ is empty. Also, we define
the global log canonical threshold of $D$ by the number
$$\mathrm{lct}(X,D)=\inf_{n\in\mathbb{N}}\ \Big\{ \mathrm{lct}_{n}(X,D)\Big\}.$$
\end{definition}

\subsection{Double spaces.}
\label{section:double-spaces} Let $\pi:V\to \mathbb{P}^n$ be a
smooth double cover  ramified along a hypersurface $S$ of degree
$2m$ in $\mathbb{P}^n$, $n\geq 3$. In addition, let $H$ be the
pull-back of a hyperplane in $\mathbb{P}^n$ by the covering map
$\pi$. We can consider the double cover $V$ as a smooth
hypersurface of degree $2m$ in $\mathbb{P}(1^{n+1},m)$.

\begin{proposition}
\label{proposition:double-spaces-equality} The global log canonical threshold
$\mathrm{lct}(V, H)$ is equal to the first global log canonical threshold $\mathrm{lct}_{1}(V, H)$.
\end{proposition}

\begin{proof} Let us use the arguments in the proof of \cite[Proposition~5]{Pu04d}.

Suppose that there is a divisor $D$ in the linear system $|\mu H|$
for some integer $\mu\geq 2$ such that
$$
c\left(V, \frac{1}{\mu}D\right)<\mathrm{lct}_{1}(V,H)
\leq 1.%
$$
It follows from
Remark~\ref{remark:weighted-sum} that we may assume that the
support of the divisor $D$ does not contain divisors of the linear
system $|H|$.

Choose a number $\lambda$ such that $c\left(V,
\frac{1}{\mu}D\right)<\lambda <\mathrm{lct}_{1}(V,H)$. Then the
log pair $\left(V, \frac{\lambda}{\mu} D\right)$ is not log
canonical. By \cite[Proposition~4.3]{Pu00} we have the center of a
non-log-canonical singularity of the log pair  $\left(V,
\frac{\lambda}{\mu} D\right)$ at a point $P$ on $V$.

Suppose that $\pi(P)\in S$. Let $T$ be the unique divisor in the
linear system $|H|$ that is  singular at the point $P$. Since we have $\mathrm{mult}_P(D)>\mu$, we
obtain an absurd  inequality
$$
2\mu=  D\cdot T^{n-1}\geq \mathrm{mult}_{P}\left(D\cap
T\right)>2\mu.
$$

Now, we suppose that $\pi(P)\not\in S$. Let $\xi\colon W\to V$ be the blow up at the point $P$ and
$E\cong\mathbb{P}^{n-1}$ be the exceptional divisor of the
birational morphism $\xi$. Then, it follows from
\cite[Proposition~3]{Pu04d} that there is a hyperplane
$\Lambda\subset E$ such that
\begin{equation*}
\label{equation:2n-inequality}
\mathrm{mult}_{P}\left(D\right)+\mathrm{mult}_{\Lambda}\left(\bar{D}\right)>2\mu,
\end{equation*}
where $\bar{D}$ is the proper transform of $D$ on the variety $W$.

Let $G$ be a general divisor in
$|H|$ such that
$\Lambda\subset\mathrm{Supp}(\bar{G})$, where $\bar{G}$ is the
proper transform of $G$ on the variety $W$. Then, we also obtain a contradictory  inequality
$$
2\mu=D\cdot G^{n-1}\geq\mathrm{mult}_{P}\left(D\cap G\right)>2\mu.
$$
\end{proof}

Now we are ready to prove the following result.

\begin{proposition}\label{proposition:double-spaces-inequality}
The following
inequality holds:
$$
\mathrm{lct}\left(V, H\right)\geq\mathrm{min}\left(1, \frac{m+n-1}{2m}\right).%
$$
\end{proposition}

\begin{proof}
By Proposition~\ref{proposition:double-spaces-equality}, it is enough to consider
the first global log canonical threshold $\mathrm{lct}_{1}(V, H)$ instead of $\mathrm{lct}(V, H)$.
Let $D$ be a divisor in
$|H|$.

The double space $V$ can be defined by a quasi-homogenous equation
$
w^{2}=f\left(x_{0},\ldots,x_{n}\right) $ in the weighted projective space
$\mathbb{P}\left(1^{n+1},m\right)\cong\mathrm{Proj}\left(\mathbb{C}[x_{0},\ldots,x_{n},w]\right)$,
where $\mathrm{wt}(x_{i})=1$, $\mathrm{wt}(w)=m$, and $f$ is a homogeneous
polynomial of degree $2m$.  Note that the homogenous polynomial $f$ defines the smooth hypersurface $S$ in $\mathbb{P}^n$ since $V$ is smooth.   We may assume that the divisor $D$
is cut out on $V$ by the equation $x_{0}=0$. The divisor $D$ is a
hypersurface in $\mathbb{P}\left(1^{n},m\right)\cong\mathrm{Proj}\left(\mathbb{C}[x_{1},\ldots,x_{n},w]\right)$ defined by the equation $w^{2}=f\left(0,x_{1},\ldots,x_{n}\right)$.
It has isolated singularities since the hypersurface
$$
D_S:=\{f\left(0,x_{1},\ldots,x_{n}\right)=0\}\subset\mathbb{P}^{n-1}\cong\mathrm{Proj}\left(\mathbb{C}[x_{1},\ldots,x_{n}]\right),
$$
has isolated singularities (see \cite{Ish82}).

It follows from \cite[Theorem~3.1]{dFEM03} that the log pair
$(V,\lambda D)$ is log terminal if and only if
$(\mathbb{P}(1^{n},m), \lambda D)$ is log terminal because $V$ is
smooth and the divisor $D$ is contained in the smooth locus of
$\mathbb{P}(1^{n},m)$. It then follows from  \cite[Proposition~8.21]{Ko97} that $$
c(V, D)=c\left(\mathbb{P}\left(1^{n},m\right), D\right)=\frac{1}{2}+c\left(\mathbb{P}^{n-1}, D_S\right).
$$
We then see that \cite[Theorem~3.1]{dFEM03} and  \cite[Theorem~3.3]{ChPa02} imply $$c\left(\mathbb{P}^{n-1}, D_S\right)=c\left(S, D_S\right)\geq
\frac{n-1}{2m}.$$
This completes the proof.
\end{proof}

Let $\pi\colon V\to\mathbb{P}^{n}$ be
a double cover ramified along a smooth
hypersurface of degree~$2n\geq 4$. It is a Fano variety of Fano index $1$ and the pull-back of a hyperplane in $\mathbb{P}^n$ is an anticanonical divisor of $V$. It follows from
Proposition~\ref{proposition:double-spaces-inequality} (for $n\geq 3$) and \cite[Theorem~1.7]{Ch07b} (for $n=2$)
that
$$
\mathrm{lct}\left(V\right)\geq \frac{2n-1}{2n},%
$$
while \cite[Theorem~2]{Pu04d} shows that $\mathrm{lct}(V)=1$ if
$V$ is general and $n\geq 3$. Therefore,
we immediately
obtain the following result that has been proved by
\cite{ArGha06} in a different way.

\begin{corollary}
\label{corollary:KE-double cover} A smooth double cover of $\mathbb{P}^n$ ramified along a hypersurface of degree $2n\geq 4$ admits a K\"ahler-Einstein metric.
\end{corollary}

\begin{remark}
\label{corollary:double-spaces}
Combining the results of \cite{Ch01b} and the proof of
Proposition~\ref{proposition:double-spaces-inequality}, we can
easily obtain the following.
Let $V$ be the smooth hypersurface
in $\mathbb{P}(1^{n+1},m)$ of degree $2m\geq 2n\geq 6$ given by an equation
$$
w^{2}=f\left(x_{0},\ldots,x_{n}\right)\subset\mathbb{P}\left(1^{n+1},m\right)\cong\mathrm{Proj}\left(\mathbb{C}[x_{0},\ldots,x_{n},w]\right),
$$
where $\mathrm{wt}(x_{i})=1$, $\mathrm{wt}(w)=m$, and $f$ is a homogeneous
polynomial of degree $2m$. Suppose that
$$
c\left(V, D\right)=\frac{m+n-1}{2m\mu},
$$
where $D\in |\mu H|$
and $\mu\in\mathbb{N}$. Then $D=\mu T$, where $T$ is a divisor
that is cut out on the hypersurface $V$ by an equation
$\sum_{i=0}^{n}\lambda_{i}x_{i}=0$ such that the hypersurface
$$
f\left(x_{0},\ldots,x_{n}\right)=\sum_{i=0}^{n}\lambda_{i}x_{i}=0\subset\mathbb{P}^{n-1}\cong\mathrm{Proj}\left(\mathbb{C}[x_{0},\ldots,x_{n}]\Big\slash \left(\sum_{i=0}^{n}\lambda_{i}x_{i}\right)\right)%
$$
is a cone over a smooth hypersurface in $\mathbb{P}^{n-2}$ of
degree $2m$.
\end{remark}

We can also give an easy proof of the following result that is a
corollary of  \cite[Theorem~2]{Pu04d}.

\begin{proposition}
\label{proposition:generic-double-spaces} Let $V$ be the double cover of $\mathbb{P}^n$, $n\geq 3$, ramified along a general hypersurface $S$ of degree $2n$ in $\mathbb{P}^n$.  Then $\mathrm{lct}(V,
H)=1$.
\end{proposition}

\begin{proof}
We assume that for every hyperplane
$M\subset\mathbb{P}^{n}$, the intersection $S\cap M$ has at most
isolated double points. This generality condition is obviously
satisfied for a general hypersurface $S$ because
$n\geq 3$.

Let $D$ be a divisor in the linear system $|H|$. It follows from
 \cite[Lemma~8.12]{Ko97} that the singularities of
the log pair $(V, D)$ are log canonical if and only if the
singularities of the log pair
$$
\left(\mathbb{P}^{n}, \pi(D)+\frac{1}{2}S\right)
$$
are log canonical. Put $M=\pi(D)$. It follows from
\cite[Theorem~7.5]{Ko97} that the singularities of the log pair
$(V, D)$ are log canonical if and only if the log pair $(M,
\frac{1}{2}S\vert_{M})$ is log canonical. But the log pair $(M,
\frac{1}{2}S\vert_{M})$ is log canonical because $S\vert_{M}$ has
at most  isolated double points.
\end{proof}

The generality assumption in
Proposition~\ref{proposition:generic-double-spaces} is weaker than
that of  \cite[Theorem~2]{Pu04d}.

Let $V$ be the double cover of $\mathbb{P}^3$ ramified along a smooth sextic $S\subset\mathbb{P}^3$. Note that the pull-back of a hyperplane in
 $\mathbb{P}^3$ is an anticanonical divisor.
As we did for quartic threefolds, we are also able to find all the
possible first global log canonical thresholds of $V$.

\begin{proposition}
\label{theorem:sextic-double-solid} Let $V$ be the smooth double
cover of $\mathbb{P}^3$ ramified along a sextic. Then,  the first
global log canonical threshold of the Fano variety $V$ is one of
the following:
$$
\{\frac{5}{6},\frac{43}{50},\frac{13}{15},\frac{33}{38},\frac{7}{8},\frac{33}{38},
\frac{8}{9}, \frac{9}{10},\frac{11}{12},\frac{13}{14},
\frac{15}{16},\frac{17}{18},
\frac{19}{20},\frac{21}{22},\frac{29}{30},1\}.$$ Furthermore, for
each number $\mu$ in the set above, there is a smooth double cover
$V$ of $\mathbb{P}^3$ ramified along a sextic with
$\mathrm{lct}_{1}(V)=\mu$.
\end{proposition}
\begin{proof}
For the proof, see \cite{W10}.
Its brief idea is as follows.
For a hyperplane $H$ in $\mathbb{P}^3$, we see that
\[c\left(V, \pi^*(H)\right)=\min\left\{1, \frac{1}{2}+c\left(H, H\cap S\right)\right\}.\]
The intersection $H\cap S$ is a reduced  sextic plane curve on
$H\cong\mathbb{P}^2$. Therefore, for the first statement of
Proposition~\ref{theorem:sextic-double-solid}, it is enough to
consider all the possible values of $c(\mathbb{P}^2, C)$ for
reduced sextic plane curves. Furthermore, we can consider only the
values for $c_0(f)$, where $f$ is a reduced sextic polynomial
vanishing at the origin.
\end{proof}

Because the first global log canonical thresholds coincide with
the global log canonical thresholds on double spaces,
Proposition~\ref{theorem:sextic-double-solid} implies a stronger
result as follows.
\begin{corollary}
\label{corollary:global-lct-sextic-double-solid} Let $V$ be a
smooth double cover of $\mathbb{P}^3$ ramified along a sextic.
Then, the  global log canonical threshold of the Fano variety $V$
is one of the numbers in
Proposition~\ref{theorem:sextic-double-solid}. Furthermore, for
each number $\mu$ in
Proposition~\ref{theorem:sextic-double-solid}, there is a smooth
double cover $V$ of $\mathbb{P}^3$ ramified along a sextic with
$\mathrm{lct}(V)=\mu$.
\end{corollary}

Let us finish the paper by an example of a smooth double cover of
$\mathbb{P}^3$ ramified along a sextic surface with the global log
canonical threshold $1$.

\begin{example}
Let $V$ be the smooth double cover of $\mathbb{P}^3$ ramified along the sextic surface $S\subset \mathbb{P}^3$ defined
by the equation
\[x_0^6+x_1^6+x_2^6+x_3^6+x_0^2x_1^2x_2x_3=0.\]
Let $C\subset \mathbb{P}^3$ be the curve defined by the intersection of the surface $S$ and the Hessian surface $\mathrm{Hess}(S)$ of $S$.
 For the tangent hyperplane $T_P$ at a point $P\in S$, if the multiplicity of the curve $T_P\cap S$
at the point $P$ is at least $3$, then the curve $C$ is singular at the point $P$.  Using the computer program, \emph{Singular}, one can check that the curve $C$
 is smooth in the outside of
the curves $x_i=x_j=0$ with $i\ne j$. Furthermore, for a point $P$ in $S$ that belongs to the curves $x_i=x_j=0$ with $i\ne j$, one can easily
check that the log pair $(S, \frac{1}{2}H_P)$ is log canonical, where $H_P$ is the hyperplane section of $S$ by
the tangent hyperplane to $S$ at the point $P$. Consequently, $\mathrm{lct}(V)=\mathrm{lct}_{1}(V)=1$. The variety $V$ is an explicit example of smooth Fano variety with the following properties (We do not know any other explicit example of such a smooth Fano variety).
For each $i=1,2, \cdots, r$, let $V_i=V$.
Then, the paper \cite{Pu04d} implies that the product $V_1\times\cdots\times V_r$ is not rational and
\[\mathrm{Bir}(V_1\times\cdots\times V_r)=\mathrm{Aut}(V_1\times\cdots\times V_r).\] Moreover, for each dominant rational map
$\rho : V_1\times\cdots\times V_r\dasharrow Y$ whose general fiber is rationally connected, there is a subset $\{i_1, \cdots, i_k\}\subset \{1, \cdots, r\}$ such that the diagram

$$
\xymatrix{
&V_1\times\cdots\times V_r \ar@{->}[d]_{\pi}\ar@{-->}[drr]^{\rho}&&&\\
&V_{i_1}\times\cdots\times V_{i_k}\ar@{-->}[rr]_{\bar{\rho}}&&Y&}
$$
commutes, where $\pi$ is the natural projection and $\bar{\rho}$ is a birational map.
\end{example}

\noindent{\bf Acknowledgements.} The authors would like to express their sincere appreciation to the referee for the invaluable comments. The referee's comments enable us to improve their results as well as their exposition.  In particular, the referee pointed out a gap in the previous proof of Lemma~\ref{lemma:mult 3}. To fix the gap, the authors introduce new generality conditions to quintic fourfolds. This serious revision was done while the first two authors  stay at Hausdorff Research Institute for Mathematics at Bonn, Germany for Research in Groups Program from 1st of August to 4th of September 2012. Ivan Cheltsov and Jihun Park would like to thank the institute for their support.
Jihun Park  has been  supported by the Research Center Program  (Grant No. CA1205-02) of Institute for Basic Science
and SRC-GAIA(Grant No. 2011-0030795)  of the National Research
Foundation in Korea.

\setcounter{equation}{0}

\renewcommand{\theequation}{\Roman{equation}}
\renewcommand{\thesection}{}

\section{Appendix}

Let $X_F$ be a smooth quintic hypersurface in $\mathbb{P}^{5}$ that
is given by zeroes of a section $F\in
H^{0}(\mathbb{P}^5,\mathcal{O}_{\mathbb{P}^{5}}(5))$. It follows
from Proposition~\ref{proposition:n-regular} that there exists a
non-empty Zariski open subset $U_{G1}\in
H^{0}(\mathbb{P}^5,\mathcal{O}_{\mathbb{P}^{5}}(5))$ such that $X_F$
is $4$-regular whenever $F\in U_{G1}$. Similarly, it follows
from Lemma~\ref{lemma:G2} that there exists a non-empty Zariski
open subset $U_{G2}\in
H^{0}(\mathbb{P}^5,\mathcal{O}_{\mathbb{P}^{5}}(5))$ such that for
every $3$-dimensional linear space $\Pi$ in $\mathbb{P}^5$, the
intersection $X_F\cap\Pi$ is irreducible and reduced if $F\in
U_{G2}$.

The purpose of this Appendix is to prove Lemma~\ref{lemma:G3},
i.e., to prove the existence of a non-empty Zariski open subset
$U_{G3}\in H^{0}(\mathbb{P}^5,\mathcal{O}_{\mathbb{P}^{5}}(5))$
such that for each $F\in U_{G3}$ the hypersurface $X_F$ satisfies the condition G\,3 (see
Section~\ref{subsection:general-quintic}). Indeed, we prove the statement as follows:
\begin{center}
 \begin{minipage}[m]{.95\linewidth}
 \emph{For each $a (=0,1,2,3) $ and $b (=1,2,\cdots,6)$,
there exists a non-empty Zariski open subset
$U$ in $H^{0}(\mathbb{P}^5,\mathcal{O}_{\mathbb{P}^{5}}(5))$
such that if $F\in U$, then for each point $P\in X$ and
each $3$-dimensional linear space $\Pi_3$ contained in the tangent
hyperplane at $P$ and containing the point $P$, the surface
$Z:=X\cap\Pi_3$ satisfies the condition G\,3.$a$.$b$.}
\end{minipage}
\end{center}
Since we use the same method in order to prove the statement for each $a$ and $b$, we first explain how the proof goes and then show the required computations in each case G\,3.$a$.$b$.

The proof goes as follows.

First we consider
the space
$$
\mathcal{S}=\mathcal{F}\times H^{0}\left(\mathbb{P}^{5},\mathcal{O}_{\mathbb{P}^{5}}\left(5\right)\right)%
$$
with the natural projections $p\colon\mathcal{S}\to
H^{0}(\mathbb{P}^5,\mathcal{O}_{\mathbb{P}^{5}}(5))$ and
$q\colon\mathcal{S}\to \mathcal{F}$.
Here, $\mathcal{F}$ is a suitable flag variety in $\mathbb{P}^5$. Depending on the case, the flag $\mathcal{F}$ will be
$Flag(0,1,2,3,4)$, $Flag(0,2,3,4)$, $Flag(0,1,3,4)$ or $Flag(0,3,4)$, where
$Flag(n_1,\cdots, n_k)$ is the flag variety that parametrizes $k$-tuples $(\Pi_{n_1},\cdots, \Pi_{n_k})$ of $n_i$-dimensional linear spaces with $\Pi_{n_1}\subset\cdots\subset \Pi_{n_k}\subset\mathbb{P}^5$ . A $0$-dimensional linear  space will be denoted by $P$ and a four dimensional linear space will be denoted by $T$.

We then put
\[\mathcal{I}=\left\{ \ \left((P,\Pi_{n_2},\cdots,\Pi_{n_{k-1}},T), F\right)\in\mathcal{S}\ \ \left| \
 \aligned
& F(P)=0;\\ & T\ \mbox{is the  tangent  hyperplane to $X_F$  at}\ P;\\
& X_F \mbox{ satisfies the properties $\mathcal{P}_{G3.a.b}$.}
\endaligned\right.\right\},
 \]
where the properties $\mathcal{P}_{G3.a.b}$ will be specified in the individual proofs.
Then in each case, we will see that it is easy to check that the morphism $q |_\mathcal{I} :\mathcal{I}\to \mathcal{F}$ is surjective.

With this set up, we compute the codimension $c$ of $q |_\mathcal{I}^{-1}(P,\Pi_{n_2},\cdots,\Pi_{n_{k-1}},T)$
in $ H^{0}(\mathbb{P}^5,\mathcal{O}_{\mathbb{P}^{5}}(5))$ for a point $(P,\Pi_{n_2},\cdots,\Pi_{n_{k-1}},T)\in\mathcal{F}$.
We may always assume that
 $T$ is defined by $x=0$, $\Pi_3$ is defined by $x=y=0$, $\Pi_2$ by $x=y=z=0$, $\Pi_1$ by $x=y=z=u=0$ and $P=[0:0:0:0:0:1]$.
We write the quintic polynomial $F$ as
\[w^5q_0+w^{4}q_{1}\left(x,y,z,u,v\right)+w^{3}q_{2}\left(x,y,z,u,v\right)+w^{2}q_{3}\left(x,y,z,u,v\right)+wq_{4}\left(x,y,z,u,v\right)+q_{5}\left(x,y,z,u,v\right),\]
where $q_{i}$ is a homogeneous polynomial of degree $i$.

The condition $F(P)=0$ is equivalent to $q_0=0$. The condition that $T$ is the tangent hyperplane to $X_F$ at $P$ is equivalent to $q_1=\lambda x$ for some $\lambda\in\mathbb{C}^*$. These two conditions contribute to the codimension $c$ by $5$.
For each $a$ and $b$, we will show that that the properties $\mathcal{P}_{G3.a.b}$ makes another contribution to the codimension $c$ by more than
$\dim\mathcal{F}-5$.

 These altogether show that the codimension of $q |_\mathcal{I}^{-1}(P,\Pi_{n_2},\cdots,\Pi_{n_{k-1}}, T)$ in $ H^{0}(\mathbb{P}^5,\mathcal{O}_{\mathbb{P}^{5}}(5))$ is  more than $\dim\mathcal{F}$.
 These implies that the morphism $p|_\mathcal{I}$ cannot be surjective. Taking the properties $\mathcal{P}_{G3.a.b}$ into consideration, we can immediately notice that this non-surjectivity implies the statement.

Therefore, to prove the statement for each case, it is enough to
\begin{itemize}
\item specify the flag $\mathcal{F}$ with its dimension;
\item specify the property $\mathcal{P}_{G3.a.b}$;
\item show that the properties $\mathcal{P}_{G3.a.b}$ makes another contribution to the codimension $c$ by more than
$\dim\mathcal{F}-5$.
\end{itemize}

Now we do these jobs for each case.

\begin{lemma}
\label{appendix:G301} The statement holds for G\,3.0.1.
\end{lemma}

\begin{proof}
The flag $\mathcal{F}$ is $Flag(0,1,3,4)$. It is of dimension 14.
Put
\[\mathcal{P}_{G3.0.1}=\left\{
 X_F\cap\Pi_3 \mbox{ is singular along $\Pi_2$.}
\right\}.
 \]
The condition that $X_F\cap\Pi_3$ contains the line $\Pi_1$ is equivalent to the condition that for each $i=2,3,4,5$, the polynomial $q_i$ contains no $v^i$. For $X_F\cap\Pi_3$ in order to be singular along $L$, for each $i=2,3,4,5$, the polynomial $q_i$ must not contain the monomials $zu^rv^{i-r-1}$, $r=0,1,\cdots, i-1$.
 These altogether show that the properties  $\mathcal{P}_{G3.0.1}$  is of codimension $> 9$.
\end{proof}

\begin{lemma}
\label{appendix:G302} The statement holds for G\,3.0.2.
\end{lemma}
\begin{proof}
The flag $\mathcal{F}$ is $Flag(0,3,4)$. It is of dimension 12.
Put
\[\mathcal{P}_{G3.0.2}=\left\{ X_F\cap\Pi_3 \mbox{ contains four lines.}\right\}.
 \]
Since we may assume that $q_1, q_2, q_3, q_4$ forms a regular sequence, $X_F\cap \Pi_3$ containing four lines is equivalent to $q_5(x,y,z,u,v)$ vanishing at four given points in $q_1(x,y,z,u,v)=q_2(x,y,z,u,v)=q_3(x,y,z,u,v)=q_4(x,y,z,u,v)=0$ in $\mathbb{P}^4$ and
$q_4(0,0,z,u,v)$ vanishing at four given points in $q_2(0,0,z,u,v)=q_3(0,0,z,u,v)=0$ in $\mathbb{P}^2$.
These altogether show that the properties  $\mathcal{P}_{G3.0.2}$  is of codimension 8.
\end{proof}

\begin{lemma}
\label{appendix:G31} The statement holds for G\,3.1.1.
\end{lemma}

\begin{proof}
The flag $\mathcal{F}$ is $Flag(0,1,3,4)$. It is of dimension 14.

Put
\[\mathcal{P}_{G3.1.1}=\left\{
 \aligned
 & X_F\cap\Pi_3 \mbox{ contains $\Pi_1$;}\\
 & \Pi_1 \mbox{  meets its residual curve by a general hyperplane}\\
&  \mbox{section of $X_F\cap\Pi_3$ in $\Pi_3$ only at singular points;}\\
&X_F\cap\Pi_3 \mbox{ has at most one ordinary double point on $\Pi_1$.}\\
\endaligned\right\}.
 \]
 We write \[q_i(0,0,z,u,v)=\sum_{r+s+t=i}A_{rst}z^ru^sv^t,\]
 where $A_{rst}$'s are constants.

The condition that $X_F\cap\Pi_3$ contains the line $\Pi_1$ is equivalent to $A_{00t}=0$ for $t=2$, $3$, $4$ and $5$ since the line $\Pi_1$ is defined by $x=y=z=u=0$.

The surface $X_F\cap\Pi_3$ has singular points on the line $\Pi_1$ exactly where the polynomials $A_{101}vw^3+A_{102}v^2w^2+A_{103}v^3w+A_{104}v^4$ and $A_{011}vw^3+A_{012}v^2w^2+A_{013}v^3w+A_{014}v^4$ have  common zeros in $\mathbb{P}^1$. The zero given by $v=0$ corresponds to the singular point $P$.  To see this, put $\bar{F}(z,u,v,w)=F(0,0,z,u,v,w)$. Since $A_{00t}=0$ for $t=2$, $3$, $4$ and $5$, we always have
$\frac{\partial\bar{F}}{\partial v}(0,0,v,w)=\frac{\partial\bar{F}}{\partial w}(0,0,v,w)=0$. The common zeros of
$$\frac{\partial\bar{F}}{\partial z}(0,0,v,w)=v(A_{101}w^3+A_{102}vw^2+A_{103}v^2w+A_{104}v^3),$$
$$\frac{\partial\bar{F}}{\partial u}(0,0,v,w)=v(A_{011}w^3+A_{012}vw^2+A_{013}v^2w+A_{014}v^3)$$
are the singular points of $X_F\cap\Pi_3$ on the line $\Pi_1$. Note that $\Pi_1$ and its residual curve by a general hyperplane meet at every singular point of $X_F\cap\Pi_3$ on the line $\Pi_1$. Therefore, the second condition is equivalent to the condition that the polynomials $A_{101}vw^3+A_{102}v^2w^2+A_{103}v^3w+A_{104}v^4$ and $A_{011}vw^3+A_{012}v^2w^2+A_{013}v^3w+A_{014}v^4$ have  four common zeros in $\mathbb{P}^1$ with counting multiplicity, i.e., these two polynomials are proportional. This imposes three additional independent conditions on the coefficients of $F$.

The condition that the polynomial $A_{101}vw^3+A_{102}v^2w^2+A_{103}v^3w+A_{104}v^4$ has $k$ zeros without counting multiplicity imposes  $4-k$ additional independent conditions on the coefficients of $F$.
Note that $1\leq k\leq 4$.

We claim that the last condition imposes $k-1$ independent conditions on the coefficients of $F$.
Here we verify the claim only for the case with $k=4$. The other cases with $k=3$ and $2$ can be verified in the same way.

We write the homogenized Hessian matrix of the polynomial
$q_2(0,0,z,u,v)+q_3(0,0,z,u,v)+q_4(0,0,z,u,v)+q_5(0,0,z,u,v)$ along the line $\Pi_1$ as follows:
\tiny
$$
\left(%
\begin{array}{ccc}
2(A_{200}w^3+A_{201}vw^2+A_{202}v^2w+A_{203}v^3) & A_{110}w^3+A_{111}vw^2+A_{112}v^2w+A_{113}v^3 & A_{101}w^3+2A_{102}vw^2+3A_{103}v^2w+4A_{104}v^3 \\
A_{110}w^3+A_{111}vw^2+A_{112}v^2w+A_{113}v^3 & 2(A_{020}w^3+A_{021}vw^2+A_{022}v^2w+A_{023}v^3) & A_{011}w^3+2A_{012}vw^2+3A_{013}v^2w+4A_{014}v^3 \\
A_{101}w^3+2A_{102}vw^2+3A_{103}v^2w+4A_{104}v^3 & A_{011}w^3+2A_{012}vw^2+3A_{013}v^2w+4A_{014}v^3  & 0 \\
\end{array}%
\right).
$$
\normalsize
Let $H(v,w)$ be the determinant of the homogenized Hessian matrix.
The condition that three of the four singular points on $\Pi_1$ is not ordinary double points is equivalent to the condition that $H(v,w)$ vanishes at three points out of the four points defined by $A_{101}w^3v+A_{102}v^2w^2+A_{103}v^3w+A_{104}v^4=0$ and $A_{011}vw^3+A_{012}v^2w^2+A_{013}v^3w+A_{014}v^4=0$ in $\mathbb{P}^1$.
We claim that it imposes three additional independent conditions on the coefficients of $F$.
To verify the claim, we put
$$A_{110}=0, \ \ A_{111}=0, \ \ A_{112}=0, \ \ A_{113}=0, \ \ A_{102}=0, \ \ A_{103}=0$$
$$A_{012}=0, \ \ A_{013}=0 \ \ A_{201}=0, \ \ A_{202}=0, \ \ A_{021}=0, \ \ A_{022}=0.$$

Since $A_{101}w^3v+A_{104}v^4=0$ and $A_{011}vw^3+A_{014}v^4=0$ defines four points in $\mathbb{P}^1$, we have $[\lambda:\mu]\in\mathbb{P}^1$ with
$\lambda(A_{101}, A_{104})=\mu(A_{011}, A_{014})$. We then see that in our restricted situation, the condition is equivalent to the condition that  $A_{101}w^3v+A_{104}v^4=0$ has three common points with
\[\left(A_{101}w^3+4A_{104}v^3\right)^2\left\{\lambda^2\left(A_{200}w^3+A_{203}v^3\right)+\mu^2\left(A_{020}w^3+A_{023}v^3\right)\right\}=0\]
in $\mathbb{P}^1$.
Since this is a condition of codimension $3$ in the restricted situation, it verifies the claim.

These altogether show that the properties  $\mathcal{P}_{G3.1.1}$  is of codimension $> 9$.
\end{proof}

\begin{lemma}
\label{appendix:G321} The statement holds for G\,3.2.1.
\end{lemma}

\begin{proof}
The flag $\mathcal{F}$ is $Flag(0,1,2,3,4)$. It is of dimension 15.
Put
\[\mathcal{P}_{G3.2.1}=\left\{
 \aligned
& X_F\mbox{ contains $\Pi_1$}; \ \ \ X_F\cap \Pi_2\mbox{ contains a line other than $\Pi_1$ passing through $P$};\\
& X_F\cap \Pi_3 \mbox{ contains four singular points other than $P$ on the two lines on $X\cap \Pi_2$}\\
& \mbox{passing through the point $P$.}\\
\endaligned\right\}.
 \]
The condition that $X_F$ contains the line $\Pi_1$ is equivalent to the condition that for each $i=2,3,4,5$, the polynomial $q_i$ contains no $v^i$. The condition that $X_F\cap\Pi_2$ contains a line other than $\Pi_1$ passing through $P$ is equivalent to the condition that $q_3(0,0,0, u,v)$, $q_4(0,0,0,u,v)$ and  $q_5(0,0,0,u,v)$ vanish  at the point other than the point given by $u=0$  in $\mathbb{P}^1$ where $q_2(0,0,0,u,v)$ vanishes.
For $X_F\cap\Pi_3$ in order to have  four singular points other than $P$ on the two lines on $X_F\cap \Pi_2$ passing through the point $P$ is a condition of codimension $4$.
These altogether show that the properties  $\mathcal{P}_{G3.2.1}$  is of codimension $11$.
\end{proof}

\begin{lemma}
\label{appendix:G322} The statement holds for G\,3.2.2.
\end{lemma}

\begin{proof}
The flag $\mathcal{F}$ is $Flag(0,1,2,3,4)$. It is of dimension 15.
Put
\[\mathcal{P}_{G3.2.2}=\left\{
 \aligned
& X_F \mbox{ contains $\Pi_1$};\ \ \
 X_F\cap\Pi_2 \mbox{ contains two lines passing through $P$};\\
& X_F\cap\Pi_3 \mbox{ has three singular points other than $P$ on $\Pi_1$};\\
& X_F\cap\Pi_3 \mbox{ has at least one singular point on $\Pi_1$ that is not an ordinary double point.}
\endaligned\right\}.
 \]
 The condition that $\Pi_1\subset X_F$ is equivalent
to the fact that each $q_{i}(x,y,z,u,v)$ does not have $v^i$
monomial, which is condition of codimension $4$. The condition
that $X_F\cap\Pi_2$ contains another line passing through the point
$P$ is equivalent to the condition that either $q_3(0,0,0,u,v)$,
$q_4(0,0,0,u,v)$ and $q_5(0,0,0,u,v)$ vanish at the points in
$\mathbb{P}^1$ where $q_2(0,0,0,u,v)/u$ vanishes, or
$q_2(0,0,0,u,v)$ is a zero polynomial and $q_3(0,0,0,u,v)$,
$q_4(0,0,0,u,v)$ and $q_5(0,0,0,u,v)$ have common root in
$\mathbb{P}^1$. Thus, the condition that $X_F\cap\Pi_2$ contains
another line passing through the point $P$ is a condition of
codimension $3$. For the surface $X_F\cap \Pi_3$ to have three
singular points on $\Pi_1$ other than $P$ is a condition of
codimension $3$. Arguing as in the proof of
Lemma~\ref{appendix:G31}, we can see that the condition that one
of the singular points of $X_F\cap\Pi_3$ on $\Pi_1$ is not an
ordinary double point is a condition of codimension $1$.
These altogether show that the properties  $\mathcal{P}_{G3.2.2}$  is of codimension $> 10$.
\end{proof}

\begin{lemma}
\label{appendix:G323} The statement holds for G\,3.2.3.
\end{lemma}

\begin{proof}
The flag $\mathcal{F}$ is $Flag(0,1,2,3,4)$. It is of dimension 15.
Put
\[\mathcal{P}_{G3.2.3}=\left\{
 \aligned
& X_F\mbox{ contains $\Pi_1$};\\
& X_F\cap \Pi_2 \mbox{ contains a line other than $\Pi_1$ passing through $P$};\\
& X_F\cap \Pi_3 \mbox{ contains two non-ordinary singular points on $\Pi_1$.}\\
\endaligned\right\}.
 \]
The condition that $X_F$ contains the line $\Pi_1$ is equivalent to the condition that for each $i=2,3,4,5$, the polynomial $q_i$ contains no $v^i$. The condition that $X_F\cap\Pi_2$ contains a line other than $\Pi_1$ passing through $P$ is equivalent to the condition that $q_3(0,0,0, u,v)$, $q_4(0,0,0,u,v)$ and  $q_5(0,0,0,u,v)$ vanish  at the point other than the point given by $u=0$  in $\mathbb{P}^1$ where $q_2(0,0,0,u,v)$ vanishes.
As in the proof of Lemma~\ref{appendix:G31}, we can see that for $X_F\cap\Pi_3$ to have  two non-ordinary singular points on $\Pi_1$ is a condition of codimension $4$.
These altogether show that the properties  $\mathcal{P}_{G3.2.3}$  is of codimension $> 10$.
\end{proof}

\begin{lemma}
\label{appendix:G324} The statement holds for G\,3.2.4.
\end{lemma}

\begin{proof}
The flag $\mathcal{F}$ is $Flag(0,1,2,3,4)$. It is of dimension 15.
Put
\[\mathcal{P}_{G3.2.4}=\left\{
 \aligned
& X_F\mbox{ contains $\Pi_1$}; \ \ \ X_F\cap \Pi_2\mbox{ contains a line other than $\Pi_1$ passing through $P$};\\
& X_F\cap \Pi_3 \mbox{ contains two singular points other than $P$ on the line $\Pi_1$};\\
& \Pi_1 \mbox{  meets its residual curve by a general hyperplane section of $X_F\cap\Pi_3$ in $\Pi_3$ }\\
&  \mbox{only at three points.}\\ \endaligned\right\}.
 \]
  We write \[q_i(0,0,z,u,v)=\sum_{r+s+t=i}A_{rst}z^ru^sv^t,\]
 where $A_{rst}$'s are constants.

The condition that $X_F\cap\Pi_3$ contains the line $\Pi_1$ is equivalent to $A_{00t}=0$ for $t=2$, $3$, $4$ and $5$.
 The condition that $X_F\cap\Pi_2$ contains a line other than $\Pi_1$ passing through $P$ is equivalent to the condition that $q_3(0,0,0, u,v)$, $q_4(0,0,0,u,v)$ and  $q_5(0,0,0,u,v)$ vanish  at the point other than the point given by $u=0$  in $\mathbb{P}^1$ where $q_2(0,0,0,u,v)$ vanishes.
For $X_F\cap\Pi_3$ in order to have  two singular points other than $P$ on the line $\Pi_1$ and  for $\Pi_1$ to meet its residual curve by a general hyperplane section of $X_F\cap\Pi_3$ in $\Pi_3$ only at three point are equivalent to the condition that the polynomials $A_{101}vw^3+A_{102}v^2w^2+A_{103}v^3w+A_{104}v^4$ and $A_{011}vw^3+A_{012}v^2w^2+A_{013}v^3w+A_{014}v^4$ have  four common zeros in $\mathbb{P}^1$ with counting multiplicity and the polynomial $A_{101}vw^3+A_{102}v^2w^2+A_{103}v^3w+A_{104}v^4$ has three zeros without counting multiplicity. This condition is of codimention $4$.
These altogether show that the properties  $\mathcal{P}_{G3.2.4}$  is of codimension $11$.
\end{proof}
\begin{lemma}
\label{appendix:G325} The statement holds for G\,3.2.5.
\end{lemma}

\begin{proof}
The flag $\mathcal{F}$ is $Flag(0,1,2,3,4)$. It is of dimension 15.
Put
\[\mathcal{P}_{G3.2.5}=\left\{
 \aligned
& X_F\mbox{ contains $\Pi_1$}; \ \ \ X_F\cap \Pi_2\mbox{ contains a line other than $\Pi_1$ passing through $P$};\\
& X_F\cap \Pi_3 \mbox{ contains  one singular points other than $P$ on the line $\Pi_1$};\\
& \mbox{either  $\Pi_1$ meets its residual curve by a general hyperplane section in $\Pi_3$
}\\
& \mbox{only at singular points or}\\
& \mbox{$\Pi_1$ meets its residual curve by a general hyperplane section in $\Pi_3$}\\
& \mbox{only at three points and $X_F\cap\Pi_3$ has a non-ordinary singular point  $P$.}\\
\endaligned\right\}.
 \]
The first two conditions imposes seven independent conditions on
the coefficients of $F$ as before. For the surface $X_F\cap \Pi_3$
to have a  singular point on $\Pi_1$ other than $P$ and plus for
$\Pi_1$ to meet its residual curve by a general hyperplane section
in $\Pi_3$ only at singular points impose at least four
independent conditions on the coefficients of $F$.  Meanwhile, for
the surface $X_F\cap \Pi_3$ to have a  singular point on $\Pi_1$
other than $P$ and plus for $\Pi_1$ to meet its residual curve by
a general hyperplane section in $\Pi_3$ only at three points
impose at least four independent conditions on the coefficients of
$F$. However, the condition that $X_F\cap\Pi_3$ has a non-ordinary
singular point  $P$ is of codimension $1$. Therefore, the
properties  $\mathcal{P}_{G3.2.5}$  is of codimension $11$.

These altogether show that the properties  $\mathcal{P}_{G3.2.5}$  is of codimension $11$.
\end{proof}

\begin{lemma}
\label{appendix:G326} The statement holds for G\,3.2.6.
\end{lemma}

\begin{proof}
The flag $\mathcal{F}$ is $Flag(0,1,2,3,4)$. It is of dimension 15.
Put
\[\mathcal{P}_{G3.2.6}=\left\{
 \aligned
& X_F\mbox{ contains $\Pi_1$}; \ \ \ X_F\cap \Pi_2\mbox{ contains a line other than $\Pi_1$ passing through $P$};\\
& \Pi_1 \mbox{  meets its residual curve by a general hyperplane section of $X_F\cap\Pi_3$ in $\Pi_3$ }\\
&  \mbox{at at most two points.}\\ \endaligned\right\}.
 \]
The first two conditions imposes seven independent conditions on the coefficients of $F$ as before.

For the last condition,
we write \[q_i(0,0,z,u,v)=\sum_{r+s+t=i}A_{rst}z^ru^sv^t,\]
 where $A_{rst}$'s are constants.

The last condition is equivalent to the condition that
either the polynomials $A_{101}vw^3+A_{102}v^2w^2+A_{103}v^3w+A_{104}v^4$ and $A_{011}vw^3+A_{012}v^2w^2+A_{013}v^3w+A_{014}v^4$ have  a common zero at $v=0$  with multiplicity at least $3$ or they are proportional and  have only two zeros (without counting multiplicities). The former and the latter are both a condition of codimension at least $4$.

Therefore, the properties  $\mathcal{P}_{G3.2.6}$  is of
codimension at least $11$.
\end{proof}

\begin{lemma}
\label{appendix:G331} The statement holds for G\,3.3.1.
\end{lemma}

\begin{proof}
The flag $\mathcal{F}$ is $Flag(0,2,3,4)$. It is of dimension 14. Put
\[\mathcal{P}_{G3.3.1}=\left\{
 \aligned
& X_F\cap\Pi_2 \mbox{ contains three lines  passing through $P$};\\
& \mbox{either $X_F\cap\Pi_3$ has a singular point  on $\Pi_2$ other than  $P$ or one of the lines $L_i$ meets }\\
& \mbox{its residual curve by a general hyperplane section in $\Pi_3$  at at most three points.}\\
\endaligned\right\}.
 \]

 The condition that $X_F\cap\Pi_2$ contains three lines passing through the point $P$ is equivalent to the condition that
$q_2(0,0,0,u,v)$ is identically zero; $q_4(0,0,0,u,v)$ and  $q_5(0,0,0,u,v)$ vanish  at the three points  in $\mathbb{P}^1$ where $q_3(0,0,0,u,v)$ vanishes. For the surface $X_F\cap \Pi_3$ to have a  singular point on $\Pi_2$ other than $P$ is a condition of codimension $1$.
For one of the lines $L_i$ to meet its residual curve by a general hyperplane section in $\Pi_3$  at at most three points is also a condition of codimension $1$. 
 These altogether show that the properties  $\mathcal{P}_{G3.3.1}$  is of codimension $> 9$.
\end{proof}

\begin{lemma}
\label{appendix:G332} The statement holds for G\,3.3.2.
\end{lemma}

\begin{proof}
The flag $\mathcal{F}$ is $Flag(0,3,4)$. It is of dimension 12.
Put
 \[\mathcal{P}_{G3.3.2}=\left\{
 \aligned
& X_F\cap\Pi_3 \mbox{ contains three lines passing through $P$};\\
& \mbox{either one of the lines $L_i$ meets its residual curve by a general hyperplane section}\\
& \mbox{in $\Pi_3$ at at most one smooth point or}\\
& \mbox{one of the lines $L_i$ meets its residual curve by a general hyperplane section in $\Pi_3$}\\
& \mbox{at at most two smooth points and $X_F\cap\Pi_3$ has a non-ordinary singular point at $P$.}\\
\endaligned\right\}.
 \]

The condition that $X_F\cap\Pi_3$ contains three lines passing through the point $P$ is equivalent to the condition that
$q_4(0,0,z,u,v)$ and  $q_5(0,0,z,u,v)$ vanish  at  three points  in $\mathbb{P}^2$ where both $q_2(0,0,z,u,v)$ and $q_3(0,0,z,u,v)$ vanish.

For one of the lines $L_i$ to meet its residual curve by a general hyperplane section in $\Pi_3$  at at most one smooth is also a condition of codimension at least $2$.
For one of the lines $L_i$ to meet its residual curve by a general hyperplane section in $\Pi_3$  at at most two smooth is also a condition of codimension at least $1$. 
 The condition that $X_F\cap\Pi_3$ has a non-ordinary singular point at $P$ is equivalent to the condition that the quadratic polynomial $q_2(0,0,z,u,v)$ is singular in variables
$z$, $u$, $v$. This is a condition of codimension $1$.
These altogether show that the properties  $\mathcal{P}_{G3.3.2}$  is of codimension $> 7$.
\end{proof}

\end{document}